\documentclass[a4paper,12pt,reqno]{amsart}
\usepackage{mymacros}
\usepackage[
    inline,
    final
    ]{showlabels}

\newcommand{\co}{\mathop{\text{co}}}
\newcommand{\qcoh}{\operatorname{qcoh}}
\newcommand{\rad}{\operatorname{rad}}
\newcommand{\characteristic}{\operatorname{char}}
\newcommand{\Comodule}{\mathop{\text{Comod}}}
\newcommand{\representations}{\operatorname{rep}}

\DeclareMathOperator*\medoplus{\mathchoice
  {\textstyle\bigoplus}
  {\textstyle\bigoplus}
  {\scriptstyle\bigoplus}
  {\scriptscriptstyle\bigoplus}
}

\title{On stacky surfaces and noncommutative surfaces}

\author{Eleonore Faber}
\thanks{Part of the work by EF was carried out at the Isaac Newton Institute in Cambridge, supported by EPSRC grant EP/R014604/1.}
\address{School of Mathematics, University of Leeds, LS2 9JT Leeds, United Kingdom}
\email{E.M.Faber@leeds.ac.uk}

\author{Colin Ingalls}
\thanks{CI was partially supported by a Discovery Grant from the
  National Science and Engineering Research Council of Canada}
\address{School of Mathematics and Statistics, Carleton University, Ottawa, ON K1S 5B6, Canada}
\email{cingalls@math.carleton.ca}

\author{Shinnosuke Okawa}
\thanks{SO was partially supported by JSPS Grants-in-Aid for Scientific Research
(18H01120,
19KK0348,
20H01797,
20H01794,
21H04994).}
\address{Department of Mathematics, Graduate School of Science, Osaka University,  Machikaneyama 1-1, Toyonaka, Osaka 560-0043, Japan}
\email{okawa@math.sci.osaka-u.ac.jp}

\author{Matthew Satriano}
\thanks{MS was partially supported by a Discovery Grant from the
  National Science and Engineering Research Council of Canada and a Mathematics Faculty Research Chair.}
\address{Department of Pure Mathematics, University
  of Waterloo, Canada}
\email{msatrian@uwaterloo.ca}

\date{}

\begin{document}

\begin{abstract}
    Let \( \bfk \) be an algebraically closed field of characteristic \( \geq 7 \) or zero.
    Let \( \cA \) be a tame order of global dimension \( 2 \) over a normal surface \( X \) over \( \bfk \) such that
    \(
        \zentrum (\cA)
        =
        \cO_X
    \)
    is locally a direct summand of \( \cA \).  We prove that there is a \( \mu_N \)-gerbe \( \cX \) over a smooth tame algebraic stack whose generic stabilizer is trivial, with coarse space \(X \) such that the category of 1-twisted coherent sheaves on \( \cX \) is equivalent 
    to the category of coherent sheaves of modules on \( \cA \). Moreover, the stack \( \cX \) is constructed explicitly through a sequence of root stacks, canonical stacks, and gerbes.  This extends results of Reiten and Van den Bergh to finite characteristic and the global situation. As applications, in characteristic \(0\) we prove that such orders are geometric noncommutative schemes in the sense of Orlov, and we study relations with Hochschild cohomology and Connes' convolution algebra.
\end{abstract}

\maketitle

\setcounter{tocdepth}{1}

\tableofcontents

%
%
\section{Introduction}

Noncommutative algebraic geometry is a burgeoning field connecting noncommutative algebra, representation theory, derived categories, and algebraic geometry.
Understanding the extent to which noncommutative objects are determined by commutative ones is of central interest and is a theme which underlies the noncommutative McKay correspondence. In particular, we are interested in understanding when given orders $\cA$ over schemes $X$, we can find stacks $\cX$ with \( \Module \cA \simeq \qcoh \cX. \)

This question can be modified by asking for \( \Module \cA \) to be a component of the category \( \qcoh \cX \),
or a semi-orthogonal component of the derived category \( \derived ^{ \bounded } \coh \cX \).  The following two examples are well known instances of this situation.

\begin{example}
    Let \( G \) be a finite group and \(R\) a commutative ring with \( G \)-action.
    Let \(R \ast G\) be the skew group ring and
    \(
        \left[
            \Spec R / G
        \right]
    \)
    be the quotient stack.  It follows immediately from the definitions that we have the following equivalence of categories
    \begin{align}
        \Module(R \ast G)
        \simeq
        \mbox{(\( G \)-equivariant \(R\)-modules)}
        \simeq
        \qcoh
        \left[
            \Spec R / G
        \right].
    \end{align}
\end{example}

\begin{example}[cf.~\cref{corollary:orthogonal decomposition}]
    Let \( \cA \) be an Azumaya algebra of order \(N\) on a scheme \( X \). Then we can associate a \( \mu_N \)-gerbe \( \cX \) to \( \cA \) such that the following equivalences of categories hold. Here, \( \qcoh^{(i)} \cX \) denotes the category of \( i \)-twisted quasi-coherent sheaves on \( \cX \).
    \begin{align}
        \module \cA^{ \otimes i} \simeq \qcoh^{(i)} \cX
    \end{align}
    \begin{align}
        \bigoplus_{i=0}^{N-1}
        \module \cA^{\otimes i}
        \simeq
        \module
        \prod_{i=0}^{N-1}
        \cA^{\otimes i}\simeq \qcoh \cX     
    \end{align}
\end{example}

Inspired by work of Connes~\cite{MR1303779}, Chan and the second author~\cite{MR2018958} showed that to every stack \( \cX \) with a choice of finite flat presentation, there is an associated noncommutative coordinate ring (the convolution algebra); this ring may be thought of as a noncommutative scheme. Furthermore, they showed that in the case of noncommutative \emph{curves}, their construction is reversible:~every noncommutative curve is the noncommutative coordinate ring of a smooth Deligne--Mumford curve~\cite[Theorem~7.7]{MR2018958}.
Finding a groupoid whose convolution algebra is a given algebra has been studied in the setting of \(C^{ \ast }\)-algebras, see for example~\cite{MR854149,MR3976582,MR4166588}.

It is desirable to have a converse to~\cite{MR2018958} beyond the curve case since it allows one to import (stacky) geometric techniques in the study of noncommutative rings. It is even further desirable to ensure that in such a converse construction, the stack is Morita equivalent to the module category of the noncommutative ring.  We say that an additive category \( \cC \) is a component of
the additive category
\( \cD \) if there is an additive category \( \cE \) such that
\( \cD \) is equivalent to the direct sum
\(\cD \simeq \cC \oplus \cE \).
The following questions are of interest:

\begin{question}
    Let \( \cA \) be a sheaf of algebras over a scheme \( S \).
    \begin{itemize}
    \item When does there exist a stack \( \cX \) such that \( \module \cA \) is a component of
    \(\qcoh \cX \)?

    \item When does there exist a stack \( \cX \) with coarse moduli space
    \(
        \pi \colon \cX \to S
    \)
    with a vector bundle \( V \) on \( \cX \) such that
    \(
        \cA
        \simeq
        \pi _{ \ast }
        \cEnd
        (
            V
        )
    \)?
  \end{itemize}
\end{question}

We achieve both of these goals in the case of noncommutative surfaces. Our result is much more involved than the case of noncommutative curves due to the presence of singularities. In this paper we consider the following class of noncommutative surfaces (see~\cref{situation:local} for its local version).

\begin{definition}\label{definition:nc-surface}
Let \( \bfk \) be an algebraically closed field of characteristic \( 0 \) or \( p \ge 7\).
We say the pair \((X,\cA)\) is a \emph{smooth tame split order of dimension \(2\)} if the following hold.

\begin{itemize}
    \item 
    \(X \) is a normal, connected, and quasi-projective surface over \( \bfk \).

    \item
    \( \cA \) is a split tame order on \( X \) of \( \gldim \cA = 2 \); this implies \( \cA \) is a reflexive coherent sheaf on \(X \). Here, split means that
    \(
       \cO _{ X }
    \)
    is locally a direct summand of
    \(
       \cA
    \)
    as an
    \(
        \cO _{ X }
    \)-submodule, which holds for example if \( \characteristic ( \bfk ) \nmid \rank \cA \). See~\cref{definition:tame order} for the definition of tameness.

    \item
    \( \zentrum ( \cA ) = \cO _{ X } \).
\end{itemize}
\end{definition}

The theorem below is our main result and is a combination of~\cref{theorem:equivalence of cA and cAr and cAc} and the main result of~\cref{section:gerbe}.

\begin{theorem}\label{theorem:main} 
Let
\(
   \bfk
\)
and the pair
\(
   \left(
    X,
    \cA
   \right)
\)
be as in~\cref{definition:nc-surface}. Then there exists a smooth tame algebraic stack \( \cX \) with coarse space \( X \) and an equivalence of
\(
   \bfk
\)-linear categories
\begin{align}\label{equation:Morita equivalence}
    \module \cA  \simeq \coh^{ ( 1 )} \cX
\end{align}
between the category of coherent right \( \cA \)-modules and \( 1 \)-twisted coherent sheaves on \( \cX \). 
Furthermore, if \(\cA\otimes \bfk(X)\) is a matrix ring over the quotient field \( \bfk ( X ) \) of \( X \), then \( \cX \) is generically a scheme and
\(
    \module \cA  \simeq \coh\cX
\).
\end{theorem}

The following corollary is an immediate application of~\cref{theorem:main}. 
To the best of the authors' knowledge, this is the only known proof so far and the problem is open if we allow the global dimension of \( \cA \) to be greater than \(2\).

\begin{corollary}\label{corollary:geometricity}
  Let \( \cA \) and \(X \) be as in~\cref{theorem:main} and assume \(X \) is projective.
  Further assume that \( \cX \) is smooth Deligne--Mumford and generically tame (this automatically holds in characteristic zero).
  Then the derived category
    \(
        \derived ^{ \bounded } \module \cA
    \)    
is a \emph{geometric noncommutative scheme} in the sense of~\cite[Definition~4.3]{MR3545926}.
Namely, it has
an admissible embedding into the derived category of a smooth projective variety over
    \(
       \bfk
    \).
\end{corollary}
Equation~\eqref{equation:Morita equivalence} implies that
\(
   \module
   \cA
\),
and hence its derived category, is equivalent to a component of an orthogonal decomposition of
\(
   \derived ^{ \bounded }
   \coh
   \cX
\).
On the other hand, our hypotheses on \( \cX \) combined with~\cite{MR2483938} tell us \( \cX \) is a global quotient. Hence \( \cX \) is a smooth projective tame algebraic stack
(see~\cite[Definition~2.5]{MR3573964} for the definition of projective stacks), so it is a geometric noncommutative scheme by~\cite[Theorem~6.4]{MR3573964}. Thus we obtain~\cref{corollary:geometricity}.

As another immediate corollary of~\cref{theorem:main}, we obtain an equivalence of Hochschild cohomology by, say,~\cite{keller2003derived}.

\begin{corollary}\label{corollary:equivalence of Hochschild cohomology}
    There is an equivalence of Hochschild cohomology
    \begin{align}\label{equation:equivalence of Hochschild cohomology}
        \HH ^{ \bullet } ( \module \cA )
        \simeq
        \HH ^{ \bullet } ( \coh ^{ ( 1 ) } \cX )
    \end{align}
\end{corollary}

The category
\(
   \coh \cX
\)
admits an orthogonal decomposition by the subcategories of
\(
   \mbar
\)-twisted sheaves for
\(
   \mbar
   \in
   \bZ / N \bZ
\)
when
\(
   \cX
\)
is a
\(
   \mu _{ N }
\)-gerbe (see~\cref{corollary:orthogonal decomposition}).
In particular we have a product decomposition of Hochschild cohomology as follows:
\begin{align}\label{equation:hochschild cohomology of gerbe}
    \HH ^{ \bullet }
    \left(
        \cX
    \right)
    =
    \prod _{ \mbar \in \bZ / N \bZ }
    \HH ^{ \bullet }
    \left(
        \coh ^{ ( \mbar ) } \cX
    \right)
\end{align}

The decomposition~\eqref{equation:hochschild cohomology of gerbe} admits the following more concrete description based on
\begin{align}
    \HH ^{ \bullet } ( \cX )
    =
    \Hom _{ \cX \times \cX } ^{ \bullet }
    \left(
        \Delta _{ \ast } \cO _{ \cX },
        \Delta _{ \ast } \cO _{ \cX }
    \right),
\end{align}
where
\(
   \Delta
   \colon
   \cX
   \to
   \cX \times \cX
\)
is the diagonal morphism. Note that
\(
   \cX \times \cX
\)
is a
\(
   \mu _{ N }
   \times
   \mu _{ N }
\)-gerbe, and since
\(
   \cO _{ \cX }
\)
is 0-twisted, the decomposition~\eqref{equation:decomposition}
of
\(
    \Delta _{ \ast } \cO _{ \cX }
\)
corresponding to the character group
\(
    \widehat{
        \mu _{ N }
        \times
        \mu _{ N }
    }
    =
    \bZ / N \bZ   
    \times
    \bZ / N \bZ   
\)
is concentrated in the anti-diagonal part as follows:
\begin{align}
    \Delta _{ \ast } \cO _{ \cX }
    =
    \bigoplus _{ \mbar \in \bZ / N \bZ }
    \left(
        \Delta _{ \ast } \cO _{ \cX }
    \right)
    _{ ( \mbar, - \mbar ) }        
\end{align}
Thus we obtain the decomposition
\begin{align}
    \Hom _{ \cX \times \cX } ^{ \bullet }
    \left(
        \Delta _{ \ast } \cO _{ \cX },
        \Delta _{ \ast } \cO _{ \cX }
    \right)
    =
    \prod _{ \mbar \in \bZ / N \bZ }
    \Hom _{ \cX \times \cX } ^{ \bullet }
    \left(
        \left(
            \Delta _{ \ast } \cO _{ \cX }
        \right)
        _{ ( \mbar, - \mbar ) },
        \left(
            \Delta _{ \ast } \cO _{ \cX }
        \right)
        _{ ( \mbar, - \mbar ) }
    \right)
\end{align}
and the
\(
   \mbar
\)-th component on the right hand side corresponds to that of the right hand side of~\eqref{equation:hochschild cohomology of gerbe}.

By definition
\(
   \coh ^{ ( 1 ) } \cX
   =
   \coh ^{ ( \overline{1} ) } \cX
\), so that via~\eqref{equation:equivalence of Hochschild cohomology} the Hochschild cohomology of
\(
   \cA
\)
is equivalent to a direct summand of the Hochschild cohomology of
\(
   \cX
\).
It is conceivable that the deformations of
\(
   \cA
\)
can be investigated by means of the Hochschild cohomology of the gerbe
\(
   \cX
\)
via~\eqref{equation:equivalence of Hochschild cohomology},
which in turn is expected to admit a version of the Hochschild--Kostant--Rosenberg isomorphism. Unfortunately, however, the HKR isomorphism for stacks is not fully established in the literature and in particular not for our gerbe \( \cX \)  (see, say,~\cite{MR4003476}, where the HKR isomorphism is proved for global quotients of smooth varieties by finite groups in characteristic \(0\)).

\vspace{1em } 

Our construction of \( \cX \) itself draws an interesting parallel between noncommutative algebraic geometry and the ``bottom-up'' construction~\cite{MR3719470} due to Geraschenko and the fourth author. Given a smooth tame Deligne--Mumford stack \( \cY \) with trivial generic stabilizer, one may consider the effective Weil divisor \( D = \sum_i e_i D_i \) on \( Y \) where the support of \( D\) is the branch divisor of the coarse space \(\pi\colon\cY\to Y\) and where \( e_i \) is the ramification degrees of \( \pi \) along \( D_i \). The bottom-up construction~\cite[Theorem~1]{MR3719470} proves that \( D \) completely determines \( \cY \). Moreover, one may reconstruct \( \cY \) from \( D \) through an iterated procedure of two simple operations known as root stacks~\cite{MR2306040,MR2450211} and canonical stacks~\cite{MR1005008}; specifically, \( \cY \) is the canonical stack over the root stack over the canonical stack of \(Y\), where the root stack is taken along \( D \). It is important to note that not every \((Y,D)\) arises as the branch divisor of a smooth Deligne--Mumford stack, see~\cite[Example~7]{MR3719470}.


In analogy with the branch divisor of the coarse space map \(\pi\colon\cY\to Y\), one may consider the branch divisor of a noncommutative surface over its underlying commutative surface, i.e.~given a smooth tame split order \((X,\cA)\) of dimension \(2\), one may consider the ramification locus \( D \) of \( \cA \). Our proof of~\cref{theorem:main} shows that \( D \) arises as the branch locus of a smooth tame stack \( \cX \). Moreover, our construction of \( \cX \) follows the bottom-up procedure described above and hence gives precise information about the geometric structure of \( \cX \). Specifically, we construct morphisms of stacks and orders
\begin{figure}[H]
    \centering
    \includegraphics[scale=1.2]{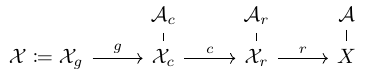}
\end{figure}
which sequentially move structure from $\cA$ into stack. More precisely, the morphisms and algebras satisfy the following properties:

\begin{enumerate}
    \item\label{item:root stack} \( \cX _{ r } \) is normal and \( \cA _{ r } \)  is a maximal order Azumaya in codimension one such that
    \begin{align}
        \module \cA _{ r } \simeq \module \cA.
    \end{align} 
   The map \( r \) resolves the singularities of the pair \( ( X, D ) \) in codimension one, where \( D \) is the
    \( \bQ \)-divisor with standard coefficients determined by \( \cA \).

    \item\label{item:canonical stack} \(\cX _{ c }\) is smooth and \( \cA _{ c } \) is the reflexive hull of \(c ^{ \ast } \cA_r\); it is an Azumaya algebra and satisfies
    \begin{align}
        \module \cA _{ c } \simeq \module \cA _{ r }.
    \end{align}

    \item~\label{item:gerb} \( g \colon \cX _{ g } \to \cX _{ c } \) is the \( \mu _{ N } \)-gerbe corresponding to \( \cA _{ c } \), where
    \(
        [ \cA _{ c } ] \in H ^{ 2 } _{ \etale } ( \cX _{ c }, \mu _{ N } )
    \).
    There is an equivalence of categories
    \begin{align}\label{equation:gerb vs Azumaya}
        \coh^{(1)}\cX _{ g } \simeq \module \cA _{ c }.
    \end{align}
\end{enumerate}

Our constructions described above are both global as well as positive characteristic generalizations of those by Reiten and Van den Bergh~\cite[Section 5]{MR978602}; the latter results are established in the case of complete local rings of surfaces in characteristic zero.
In the case considered in~\cite{MR978602}, the extensions from \(X \) to \(\cX _{ r }\) and from \( \cX _{ r }\) to \( \cX _{ c }\) are both Galois (when interpreted appropriately). In~\cite{MR978602} they prove that the composition of extensions is again Galois, which allows them to directly show the Morita equivalence between \( \cA \) and \( \cA _{ c } \).
On the contrary, in our situation the extensions are not necessarily Galois but are quotients by finite linearly reductive group schemes. In particular, it is not clear whether the composition of the extensions is also such a quotient.
In this paper we use the Galois theory of Hopf algebras as explained in~\cite{Montgomery}, and instead of trying to combine the extensions into one, we prove the Morita equivalence in two steps; namely, we prove the local version of the equivalence between \( \cA \) and \( \cA _{ r } \) first in~\cref{corollary:Lambda1 vs Lambda} and then (the local version of) the equivalence between \( \cA _{ r } \) and \( \cA _{ c } \) in~\cref{corollary:Gamma vs Lambda1}.

%
%
\section*{Acknowledgments}
This project was initiated at the American Institute of Mathematics workshop ``Noncommutative surfaces and Artin's conjecture.'' We thank the American Institute of Mathematics for its warm hospitality. It is a pleasure to also thank Pieter Belmans, Daniel Chan, Valery Lunts, Sid Mathur, and Michel Van den Bergh.

On behalf of all authors, the corresponding author states that there is no conflict of interest.

%
%
\section{Generalizing Reiten-Van den Bergh to characteristic \( p \ge 7 \)}\label{section:Reiten-Van den Bergh}

Our goal in this section is to generalize some of the results established in~\cite{MR978602} from characteristic \( 0 \) to characteristic \( p \ge 7 \). We begin with some preliminaries.

First, throughout this paper we consider \emph{right} modules unless otherwise stated.
For an algebra \( \cA \) (over a symmetric monoidal category), we let \( \cA ^{ e } \coloneqq \cA ^{ \op } \otimes \cA \) denote the enveloping algebra. An \( \cA \)-bimodule is a module over \( \cA ^{ e } \).
We let \( \zentrum ( \cA ) \) denote the \emph{centre} of the algebra \( \cA \).

Next, Azumaya algebras will play an important role in this paper; a standard reference for this material is~\cite[Chapter~4]{MR559531}. We recall the following definition.

\begin{definition}
    Let \( R \) be a commutative local ring. An \emph{Azumaya algebra} over \( R \) is an \( R \)-algebra \( A \) which is a free \( R \)-module of finite rank such that the following natural homomorphism of \( R \)-algebras is an isomorphism.
    \begin{align}\label{equation:Azumaya}
        A \otimes _{ R } A ^{ \op } \to \End _{ R } ( A ); \quad a \otimes b ^{ \op } \mapsto \left[ x \mapsto a x b \right]
    \end{align}
    More generally, if \( \cX \) is an algebraic stack and \( \cA \) is a coherent sheaf of \( \cO _{ \cX } \)-algebras, we say \( \cA \) is a sheaf of Azumaya algebras if \( \cA \) is locally free as an \( \cO _{\cX } \)-module and the following natural map, similar to~\eqref{equation:Azumaya}, is an isomorphism.
    \begin{align}\label{eq:Azumaya condition}
        \cA \otimes _{ \cO _{ \cX } } \cA ^{ \op } \to \cEnd _{ \cO _{ \cX } } \left( \cA \right)
    \end{align}
\end{definition}

\begin{remark}\label{remark:local criterion for Azumaya algebras}
    The Azumaya condition can checked locally in the following sense.
    Let \( \cA \) be a sheaf of algebras over a stack \( \cX \), and let
    \(
        f \colon U \to \cX
    \)
    be a flat surjective morphism of stacks; note that \(f\) is not assumed to be locally finitely presented. 
    Then by faithfully flat descent, \( \cA \) is a sheaf of Azumaya algebras if and only if the pullback
    \(
        f ^{ \ast } \cA
    \)
    is.
\end{remark}

We consider the following class of noncommutative surfaces, which generalizes the set-up in Reiten-Van den Bergh~\cite{MR978602}.

\begin{situation}[{Local version of~\cref{definition:nc-surface}}]\label{situation:local}
    \begin{itemize}
    	\item[{}] 
        \item \( \bfk \) is an algebraically closed field of characteristic \( 0 \) or \( p \ge 7 \).
        \item \( R \) is an integrally closed noetherian complete local \( \bfk \)-algebra of dimension \( 2 \) such that \( R / \frakm = \bfk \). Let \( K \coloneqq Q ( R )\) be the field of fractions.
        \item \( A \) is a central simple algebra over \( K \).
        \item \( \Lambda \) is a tame \( R \)-order in \( A \) of global dimension \( 2 \).
        \item \( R \) is a direct summand of \( \Lambda \) as an \( R \)-module. This is the case, for example, when \( \characteristic ( \bfk ) \not\vert \rank \Lambda \).
    \end{itemize}
\end{situation}

The notion of tameness appearing in~\cref{situation:local} is defined as follows.

\begin{definition}\label{definition:tame order}
    Let \( ( R, A ) \) be as in~\cref{situation:local}. An \( R \)-order \( \Lambda \) in \( A \) is \emph{tame} if it is reflexive as an \( R \)-module and \( \Lambda \otimes _{ R } R _{ \frakp } \) is a hereditary \( R _{ \frakp } \)-algebra for any prime ideal \( \frakp \subseteq R \) of height \( 1\); i.e., it is an \( R _{ \frakp } \)-order of global dimension \(1\).
\end{definition}

We note that the above conditions follow automatically when $A$ has Krull dimension and global dimension two:

\begin{theorem}[\cite{MR0719665}]
  Let $A$ be an $R$ order and suppose that the Krull dimension of $A$ is the global dimension of $A$.  Then $A$ is a Cohen-Macaulay $R$ module and the localization of $A_p,$ at a prime ideal $p$ of $R$, has Krull dimension and global dimension equal to the height of $p$.
\end{theorem}
  
\begin{remark}\label{remark:direct summand assumption}
    It would be interesting to know if the last assumption of~\cref{situation:local} follows from the rest when \( \characteristic ( \bfk ) \ge 7 \).
\end{remark}

\begin{definition}
    An \( R \)-order \( \Lambda \) is \emph{of finite representation type} if there are only finitely many finitely generated indecomposable reflexive \( \Lambda \)-modules up to isomorphisms.
\end{definition}

\begin{remark}\label{remark:gldim 2 implies FRT}
    A tame order of global dimension \( 2\) is of finite representation type (see~\cite[p.~12]{MR978602}).
    In fact, by~\cite[Theorem~1.10]{MR245572}, any reflexive \( \Lambda \)-module is projective.  Alternatively, see~\cite[Proposition~2.17]{MR3251829}.
\end{remark}

We next recall the dualizing bimodule following~\cite[pp.~662--663]{MR2492474}. Consider the enveloping algebra
\( \Lambda ^{ e } = \Lambda ^{ \op } \otimes _{ \bfk } \Lambda \) of \( \Lambda \) and the derived category
\( 
    \derived ( \Lambda ^{ e } )
\)
of right \( \Lambda ^{ e } \)-modules.

\begin{definition}\label{definition:rigid dualizing comple and dualizing bimodules}
    The \emph{rigid dualizing complex} of \( \Lambda \) is a pair consisting of an object \( D _{ \Lambda } \in \derived ( \Lambda ^{ e } ) \) and an isomorphism
    \(
        \chi \colon D _{ \Lambda } \simto \RHom _{ \Lambda ^{ e } } ( \Lambda, D _{ \Lambda } \otimes D _{ \Lambda } )
    \)
    in
    \(
        \derived ( \Lambda ^{ e } )
    \)
    which is uniquely characterized by the following properties:
    \begin{itemize}
        \item \( D _{ \Lambda } \) has finite injective dimension on both sides.
        \item For each \( i \), \( \cH ^{ i } ( D _{ \Lambda } ) \) is finitely generated both as a left
        \(
            \Lambda
        \)-module and as a right \( \Lambda \)-module.
        \item
        Both of the natural maps
        \(
            \Lambda \to \RHom _{ \Lambda } ( D _{ \Lambda }, D _{ \Lambda } )
        \)
        and
        \(
            \Lambda \to \RHom _{ \Lambda ^{ \op } } ( D _{ \Lambda }, D _{ \Lambda } )
        \)
        are isomorphisms in \( \derived ( \Lambda ^{ e } ) \).
    \end{itemize}

    The \emph{dualizing bimodule} \( \omega _{ \Lambda } \) is defined as follows.
    \begin{align}
        \omega _{ \Lambda } = \cH ^{ - \dim R } ( D _{ \Lambda } )
    \end{align}
\end{definition}

\begin{remark}\label{remark:bimodule}
    \begin{itemize}
        	\item[{}] 
        \item Rigid dualizing complexes exist for algebras \( \Lambda \) as in~\cref{situation:local} by~\cite[Proposition~5.7]{MR1753810}. Hence
        \(
            \omega _{ \Lambda }
            =
            \cH ^{ - 2 } ( D _{ \Lambda } )
        \)
        exists as well.
        \item We have the following explicit descriptions for \( D _{ \Lambda } \) and \( \omega _{ \Lambda } \). The proof of~\cite[Lemma~2.5]{MR2492474} applies to the \( \Lambda \) under consideration. Below \( D _{ R } \simeq \omega _{ R } [ 2 ] \) denotes the rigid dualizing complex of \( R \) (recall that \( R \) as in~\cref{situation:local} is Cohen-Macaulay).
        \begin{align}
            D _{ \Lambda } \simeq \RHom _{ R } ( \Lambda, D _{ R } )\\
            \omega _{ \Lambda } \simeq \Hom _{ R } ( \Lambda, \omega _{ R } )\label{equation:omega_Lambda explicit}
        \end{align}	
        \item
        By the uniqueness of the balanced dualizing complexes, both they and dualizing bimodules commute with localizations to (Zariski/ \'etale/ formal) neighborhoods and glue together if they exist locally. In particular the dualizing bimodules in the global~\cref{definition:nc-surface}, which are defined similarly to~\cref{definition:rigid dualizing comple and dualizing bimodules}, exist and are given by the following formula.
        \begin{align}\label{equation:dualizing bimodule of cA}
            \omega _{ \cA } = \cHom _{ \cO _{ X } } ( \cA, \omega _{ X } )
        \end{align}
    \end{itemize}
\end{remark}

\begin{definition}\label{definition:Rees algebra construction}
    In~\cref{situation:local}, let \( I \subset A \) be a \emph{divisorial ideal} of
    \(
       \Lambda
    \)
    (cf.~\cite[p.~1, p.~63]{MR978602}); i.e., a finitely generated \( \Lambda \)-submodule of \( A \) such that \( I K = A \) which is reflexive as an \( R \)-module and
    \(
        I _{ \frakp }
    \)
    is invertible for any prime ideal
    \(
        \frakp \subset R
    \)
    of height \( 1 \). We associate to \( I \) the following \( \bZ \)-graded \( A \)-algebra
    \begin{align}
        \Lambda [ I ]
        \coloneqq
        \bigoplus
        _{
            j \in \bZ
        }
        I ^{ ( j ) }
        x ^{ j }
        \subseteq
        A [ x, x ^{ - 1 } ],
    \end{align}
    where
    \begin{align}
        I
        ^{
            ( j )
        }
        &
        \coloneqq
        \begin{cases}
            \left(
                I ^{
                    \otimes
                    j
                }
            \right)
            ^{
                \vee
                \vee
            }
            &
            j
            \ge
            0
            \\
            \left[
                \Lambda
                \colon
                I
            \right]
            ^{
                ( - j )
            }
            &
            j
            <
            0
        \end{cases}\\
        \left[
            \Lambda
            \colon
            I
        \right]
        &
        \coloneqq
        \left\{
           a \in A
           \mid
           a I
           \subseteq
           \Lambda
        \right\}.
    \end{align}
    Here we opted for the notation
    \(
       \Lambda [ I ]
    \)
    over
    \(
        \Lambda
        [
            I
            ^{
                \pm
            }
        ]       
    \),
    following~\cite[p.~63]{MR978602}.

    If there is \( n > 0 \) such that
    \(
        I ^{ ( n ) }
        =
        \Lambda a
        \left(
            =
            a
            \Lambda
        \right)
    \)
    for some
    \(
        a \in K ^{ \ast }
    \), set
    \begin{align}
        {\Lambda [ I ]} _{ n } \coloneqq \Lambda [ I ] / ( 1 - a x ^{ n } ).
    \end{align}
\end{definition}

Note that
\(
    {\Lambda [ I ]} _{ n }
\)
depends on the choice of the element
\(
   a \in K ^{ \ast }
\).
Note that if $(a,n)$ solves this equation, then $(a^k,nk)$ provides another solution.  So we choose $(a,n)$ minimal such that $I^{(n)} = a \Lambda$ with respect to the  partial order $(a,n) \leq (b,m)$ if $n|m$ and $a^{\frac{m}{n}}\Lambda = b\Lambda$.
(\cite[p.~63]{MR978602}).


Applying the above construction to the divisorial ideal
\(
    I =
    D
    \left(
        \Lambda
        /
        R
    \right)
\)
defined in~\cite[p.~63]{MR978602}
(see also~\cref{remark:dualizingStuff} for details),
where the existence of \( a \) and \( n > 0 \) as above is explained in the first paragraph of~\cite[p.~67]{MR978602},
we obtain
\begin{align}
    \Lambda _{ 1 }
    &
    \coloneqq
    {\Lambda [
        D
        \left(
            \Lambda
            /
            R
        \right)
    ]} _{ n }.
\end{align}
As we explain in~\cref{remark:dualizingStuff} below, the dualizing bimodule
\(
    \omega _{ \Lambda } ^{ - 1 }
\)
is isomorphic to
\(
    D
    \left(
        \Lambda
        /
        R
    \right)
\)
as
\(
   \Lambda
\)-bimodule. In this paper we always think of
\(
   \omega _{ \Lambda } ^{ - 1 }
\)
as a divisorial ideal via this identification. In particular we use the following notation as well.
\begin{align}\label{equation:definition of Lambda1}
    \Lambda _{ 1 }
    &
    = {\Lambda [ \omega _{ \Lambda } ^{ - 1 } ]} _{ n }\\
    R _{ 1 }
    &
    \coloneqq \zentrum ( \Lambda _{ 1 } )
\end{align}

\begin{remark}\label{remark:dualizingStuff}
    It turns out that the divisorial ideal
    \(
        D
        \left(
            \Lambda
            /
            R
        \right)
    \)
    defined in~\cite[p.~63]{MR978602} is isomorphic to
    \( \omega _{ \Lambda } ^{ - 1 } \)
    as
    \(
       \Lambda
    \)-bimodule.
    Indeed, since both sides are reflexive, it is enough to confirm the assertion locally at all height \( 1 \) primes \( \frakp \) of \( R \) and use the structure theorem for the hereditary order \( \Lambda _{ \frakp } \).
    Note that the proof of~\cite[Proposition~2.7]{MR2492474} gives us the formula
    \[\omega_{\Lambda} \simeq \omega_R \otimes \rad \Lambda \frakp^{-1}\]
    in codimension one at \( \frakp \). Let \(e\) be the ramification index of \( \Lambda \) at \( \frakp \), i.e.~\(e\geq1\) is the smallest integer such that 
    \[\rad \Lambda^e =\frakp.\]
    Note furthermore that \(e|\rank \Lambda \).
\end{remark}

Although~\cite{MR978602} discusses the \( \bZ / n \)-graded algebra \( \Lambda _{ 1 } \) and its centre \( R _{ 1 } \), we could instead consider the \( \bZ \)-graded versions as follows.
\begin{align}
    \Lambdatilde
    &
    \coloneqq \Lambda [ \omega _{ \Lambda } ^{ - 1 } ]\\
    \Rtilde
    &
    \coloneqq \zentrum ( \Lambdatilde )
\end{align}

In~\cref{proposition:Z-graded vs Z/n-graded} we show that
\(
   \Lambda _{ 1 }
\)
and
\(
   \Lambdatilde
\)
are graded Morita equivalent. \cref{lemma:Lambdatilde as an invariant ring} below plays a central role in its proof. Consider the
\(
    (\bZ / n \times \bZ)
\)-algebra
\(
    \Lambda _{ 1 } [ \xi, \xi ^{ - 1 } ]
\), where the \( \bZ \)-degree of \( \Lambda _{ 1 } \) is taken to be \( 0 \) and
\begin{align}
    \deg \xi
    =
    \left( \overline{ - 1 }, 1 \right) \in \bZ / n \times \bZ.
\end{align}

\begin{lemma}\label{lemma:Lambdatilde as an invariant ring}
    There is an isomorphism of \( \bZ \)-graded algebras
    \begin{align}
        {
            \left(
                \Lambda _{ 1 } [ \xi, \xi ^{ - 1 } ]
            \right)
        }
        _{ \overline{ 0 } }
        \simeq
        \Lambdatilde,
    \end{align}
    where the left hand side is the degree \( \overline{ 0 } \) part with respect to the
    \(
        \bZ / n
    \)-grading.
\end{lemma}

\begin{proof}
    To avoid notational complication, put
    \begin{align}
        \Delta \coloneqq \Lambda _{ 1 } = \bigoplus _{ \ibar \in \bZ / n } \Delta _{ \ibar }.
    \end{align}
    Then
    \begin{align}
        {
            \left( \Delta [ \xi, \xi ^{ - 1 } ] \right)
        }
        _{ \overline{ 0 } }
        =
        \bigoplus _{ j \in \bZ } \Delta _{ \jbar } \xi ^{ j }
        =
        \bigoplus _{ \ibar \in \bZ / n }
        \bigoplus _{ \substack{ j \in \bZ \\ \overline{ j } = \ibar } }
        \Delta _{ \ibar } \xi ^{ j }.
    \end{align}

    Define the isomorphism
    \(
        \varphi \colon
        \left( \Delta [ \xi, \xi ^{ - 1 } ] \right) _{ \overline{ 0 } }
        \to
        \Lambdatilde
    \)
    as follows. For
    \(
        0 \le \ell < n
    \)
    and
    \(
        q \in \bZ
    \), put
    \(
        j = \ell + q n
    \). Let
    \(
        q _{ \ell } \colon \Lambdatilde _{ \ell } \simto \Delta _{ \ellbar }
    \)
    be the natural bijection. Then define the degree \( j \) component of \( \varphi \) as follows.
    \begin{align}
        \varphi _{ j } \colon \Delta _{ \jbar } \xi ^{ j } \to \Lambdatilde _{ j };
        \quad
        \lambda \xi ^{ j } \mapsto q _{ \ell } ^{ - 1 } ( \lambda ) a ^{ q } x ^{
            q
            n
        }
        =
        q _{ \ell } ^{ - 1 } ( \lambda ) a ^{ \lfloor \frac{ j }{ n } \rfloor } x ^{
            q
            n
        }
    \end{align}
    One can easily confirm that \( \varphi \) is a homomorphism of \( \bZ \)-graded algebras.
    In order to see that \( \varphi \) is an isomorphism, note that
    \begin{align}
        \Lambdatilde _{ j }
        =
        \Lambdatilde _{ \ell } \left( \Lambdatilde _{ n } \right) ^{ q }
        =
        \Lambdatilde _{ \ell } a ^{ q }.
    \end{align}
\end{proof}

\begin{corollary}\label{corollary:Rtilde as invariant ring}
    \begin{align}
        \left( R _{ 1 } [ \xi, \xi ^{ - 1 } ] \right) _{ \overline{ 0 } }
        \simeq
        \Rtilde.
    \end{align}
\end{corollary}

\begin{proof}
    This follows from~\cref{lemma:Lambdatilde as an invariant ring} 
    and the following computation.
    \begin{align}
        \zentrum \left( \left(\Lambda _{ 1 } [ \xi, \xi ^{ - 1 } ]\right) _{ \overline{ 0 } } \right)
        \simeq
        \left( \zentrum \left(\Lambda _{ 1 } [ \xi, \xi ^{ - 1 } ] \right) \right) _{ \overline{ 0 } }
        =
        \left( \zentrum \left( \Lambda _{ 1 } \right) [ \xi, \xi ^{ - 1 } ] \right) _{ \overline{ 0 } }
    \end{align}
\end{proof}

\begin{corollary}\label{corollary:Rtilde is of finite type}
    \( \Rtilde \) is an algebra of finite type over \( R \), and \( \Lambdatilde \) is finitely generated as an \( \Rtilde \)-module.
\end{corollary}

\begin{proof}
    The first assertion follows from~\cref{corollary:Rtilde as invariant ring}. To see this, since \( \mu _{ n } \) is reductive, it is enough to show that
    \(
        R _{ 1 } [ \xi, \xi ^{ - 1 } ]
    \)
    is of finite type over \( R \). This in turn follows from the fact that \( R _{ 1 } \) is a finitely generated \( R \)-module, as it is an \( R \)-submodule of the \( R \)-module \( \Lambda _{ 1 } \), which is obviously finitely generated.

    The second assertion similarly follows from~\cref{lemma:Lambdatilde as an invariant ring}. In fact, as \( \Lambda _{ 1 } \) is a coherent \( R _{ 1 } \)-module, \( \Lambda _{ 1 } [ \xi, \xi ^{ - 1 } ] \) is coherent as an \( R _{ 1 } [ \xi, \xi ^{ - 1 } ]\)-module, which in turn is a coherent \( \Rtilde \)-module by~\cref{corollary:Rtilde as invariant ring}.
    Since \( \Lambdatilde \) is an \( \Rtilde \)-submodule of \( \Lambda _{ 1 } [ \xi, \xi ^{ - 1 } ] \) by~\cref{lemma:Lambdatilde as an invariant ring}, it is also a coherent \( \Rtilde \)-module.
\end{proof}

\begin{corollary}\label{corollary:iso of tame algebraic stacks}
    There is an isomorphism of algebraic stacks
    \begin{align}\label{equation:iso of tame algebraic stacks}
        \left[ \Spec \Rtilde / \bGm \right]
        \simeq
        \left[ \Spec R _{ 1 } / \mu _{ n } \right].
    \end{align}
\end{corollary}

\begin{proof}
    By~\cref{corollary:Rtilde as invariant ring}, both sides are isomorphic to the quotient stack
    \begin{align}
        \left[ \Spec R _{ 1 } [ \xi, \xi ^{ - 1 } ] / \left( \mu _{ n } \times \bGm \right) \right].
    \end{align}
    To see this, note that the actions of the subgroups
    \(
        \mu _{ n }
    \)
    and \( \bGm \) on \( \Spec R _{ 1 } [ \xi, \xi ^{ - 1 } ] \) are both free.
\end{proof}

The isomorphism of algebraic stacks~\eqref{equation:iso of tame algebraic stacks} induces an equivalence of categories of coherent algebras as follows.
\begin{align}\label{equation:equivalence of categories of coherent algebras}
    \Alg \left( \coh \left[ \Spec \Rtilde / \bGm \right] \right)
    \simeq
    \Alg \left( \coh \left[ \Spec R _{ 1 } / \mu _{ n } \right] \right).
\end{align}
Note that \( \Lambdatilde \) can be regarded as an object of the left hand side of~\eqref{equation:equivalence of categories of coherent algebras}, whereas \( \Lambda _{ 1 } \) can be regarded as an object of the right hand side. They correspond to each other under~\eqref{equation:equivalence of categories of coherent algebras}.
In particular, the \( \bZ \)-graded algebra \( \Lambdatilde \) and the \( \bZ / n \)-graded algebra
\(
    \Lambda _{ 1 }
\)
are ``Morita equivalent'' as follows.

\begin{proposition}\label{proposition:Z-graded vs Z/n-graded}
    There is an equivalence of categories
    \begin{align}
        \alpha \colon
        \module ^{ \bZ } \Lambdatilde
        \rightleftarrows
        \module ^{ \bZ / n } \Lambda _{ 1 } \colon \beta
    \end{align}
\end{proposition}

\begin{proof}
    The definitions of \( \alpha \) and \( \beta \) follow from the fact that \( \Lambdatilde \) and \( \Lambda _{ 1 } \) correspond to each other under the equivalence of categories~\eqref{equation:equivalence of categories of coherent algebras}.
    For the convenience of the reader, we write them down explicitly.

    For \( N \in \module ^{ \bZ } \Lambdatilde \)
    \begin{align}
        \alpha ( N )
        \coloneqq
        \left( N \otimes _{ \Lambdatilde } \Lambda _{ 1 } [ \xi, \xi ^{ - 1 } ] \right) _{ 0 }
        \in
        \module ^{ \bZ / n } \Lambda _{ 1 }.
    \end{align}
    Conversely, given
    \(
        M \in \module ^{ \bZ / n } \Lambda _{ 1 }
    \),
    \begin{align}
        \beta ( M ) \coloneqq
        \left( M [ \xi, \xi ^{ - 1 } ] \right) _{ \overline{ 0 } }
        \in
        \module ^{ \bZ } \left( \Lambda _{ 1 } [ \xi, \xi ^{ - 1 } ] \right) _{ \overline{ 0 } }
        \stackrel{\text{\cref{corollary:Rtilde as invariant ring}}}{ \simeq }
        \module ^{ \bZ } \Lambdatilde.
    \end{align}
\end{proof}

\begin{proposition}\label{proposition:Lambda1 * mun simeq End Lambda1}
  \[\Lambda_1 \ast \cO_{\mu_n}^\vee \simeq \End_\Lambda(\Lambda_1).\]
\end{proposition}
\begin{proof}  
  First note that \(\cO_{\mu_n} = k[t]/(t^n-1)\) with comultiplication
  \(\Delta(t) = t \otimes t\), and \( \Lambda_1\) is a \(\cO_{\mu_n}\)-comodule.
  Hence \( \Lambda_1\) is a \( \cO_{\mu_n}^\vee\)-module.  The Hopf algebra \(\cO_{\mu_n}^\vee \simeq \cO_{\bZ/n}\) and so there is a complete set of orthogonal idempotents \(e_0,\ldots,e_{n-1}\) which give a basis of \(\cO_{\bZ/n}\) with pairing \(e_i(t^j) = \delta_{ij}\).  Let \( \Lambda = \sum \lambda_i \in \Lambda_1\) be an element of \( \Lambda_1\) with its decomposition into graded components.  Then action of \(\cO_{\bZ/n}\) is described by \(e_i \lambda = \lambda_i\).
  Let \(d \in \Lambda_1\) be homogeneous of degree \(j\), then
  \[e_i(d\lambda) = d\lambda_{i-j} = d e_{i-j} \lambda\]
  hence we see that
  \[\Lambda_1 \ast\cO_{\mu_n}^\vee = \bigoplus \Lambda_1 e_i\]
  with relations \(e_i d = d e_{i-j}\) for all \(d \in \Lambda_1\) of degree \(j\)
  as in~\cite[Example 4.1.7]{Montgomery}.
  Hence
  \[e_i \Lambda_1 e_j = \omega_\Lambda^{(i-j)}e_j\]
  \[e_i \Lambda_1 e_j e_j \Lambda_1 e_\ell = \omega_\Lambda^{(i-j)}\omega_\Lambda^{(j-\ell)} e_\ell\]
  and so the multiplication in \( \Lambda_1 \ast \cO_{\mu_n}^\vee\) 
  is isomorphic to that of
  \[\End_\Lambda(\medoplus \omega_\Lambda^{(i)}) = \bigoplus_{i,j} \Hom_\Lambda(\omega_\Lambda^{(i)},\omega_\Lambda^{(j)})=\bigoplus_{i,j} \omega_\Lambda^{(j-i)}  .\]  
\end{proof}

\begin{corollary}\label{corollary:Lambda1 vs Lambda}
    There is an equivalence of categories
    \begin{align}
        \module ^{ \bZ / n } \Lambda _{ 1 }
        \simeq
        \module \Lambda.
    \end{align}
\end{corollary}

\begin{proof}
    The equivalence is obtained as the composition of the following equivalences.
    \begin{align}
        \module ^{ \bZ / n } \Lambda _{ 1 }
        \simeq
        \module \Lambda _{ 1 } \ast \mu _{ n }
        \stackrel{ \text{ \cref{proposition:Lambda1 * mun simeq End Lambda1} } }{ \simeq }
        \module \End _{ \Lambda } \Lambda _{ 1 }
        \simeq
        \module \Lambda
    \end{align}
    The last equivalence follows from the fact that \( \Lambda \) has global dimension \( 2 \). In fact, \( \Lambda _{ 1 } \) as a \( \Lambda \)-module, contains \( \Lambda \) as a direct summand and is a projective module by~\cite[Theorem~1.10]{MR245572}.  Alternatively, see~\cite[Proposition~2.17~(2)\(\Rightarrow\)(3)]{MR3251829}.
\end{proof}

\begin{corollary}\label{corollary:Lambda1 has global dimension 2}
    \( \Lambda _{ 1 } \) has global dimension \( 2 \).
\end{corollary}

\begin{proof}
    By~\cref{corollary:Lambda1 vs Lambda}, since \( \Lambda \) has global dimension \( 2 \), the category
    \(
        \module ^{ \bZ / n } \Lambda _{ 1 }
    \)
    has global dimension \( 2 \). As
    \(
        \module ^{ \bZ / n } \Lambda _{ 1 }
        \simeq
        \module \Lambda _{ 1 } \ast \mu _{ n }
    \),
    the algebra \( \Lambda _{ 1 } \ast \mu _{ n } \) also has global dimension \(2\). 
    As \( \Lambda _{ 1 } \) is a direct summand of this algebra as a bimodule over itself and \( \Lambda _{ 1 } \ast \mu _{ n } \) is a free module over \( \Lambda _{ 1 } \), it follows from~\cite[7.2.8~Theorem~(i)]{MR1811901} that the global dimension of \( \Lambda _{ 1 } \) is at most \( 2 \). It is clearly at least \( 2 \), hence the conclusion.
\end{proof}

Let \( R _{ 1 } \) be the centre of \( \Lambda _{ 1 } \). It inherits the \( \bZ / n \)-grading of \( \Lambda _{ 1 } \).

Before stating the next result, we recall (and extend) the definition of linearly reductive quotient singularities.

\begin{definition}\label{df:linearly reductive quotient singularity}
    Let \( X \) be a normal surface over \( \bfk \). A closed point
    \( x \in X ( \bfk ) \) is a \emph{linearly reductive quotient singularity} if
    there is a finite linearly reductive group scheme \( H \) acting linearly on
    \(
        \bA ^{ 2 }
    \)
    such that
    \begin{align}
        \widehat{\cO}_{ X, x }
        \simeq
        \widehat{\cO} _{\bA ^{ 2 } / H, 0}.
    \end{align}
    By~\cite[Corollary~(2.6)]{MR268188}, this is equivalent to the existence of a smooth affine variety \( U \) over
    \(
        \bfk
    \), a finite linearly reductive group scheme \( H \) over
    \(
        \bfk
    \)
    acting on \( U \), and an \'etale morphism
    \begin{align}\label{equation: U / H to X}
        U / H \to X
    \end{align}
    whose image contains
    \(
        x
        \in
        X
    \).

    More generally, let \( X \) be a normal algebraic stack over a field \( \bfk \) with finite inertia and \(2\)-dimensional coarse space. We say a
    \(
        \bfk
    \)-rational point
    \(
        x \colon \Spec \bfk \to X
    \)
    is a linearly reductive quotient singularity if there exist
    \(
        H
        \curvearrowright
        U
    \)
    as above and a \emph{smooth} (not necessarily \'etale) morphism as in~\eqref{equation: U / H to X}
    whose image contains
    \(
        x
        \in
        X
    \).
\end{definition}
        
\begin{remark}\label{rmk:canonicalStackConstruction}
For schemes, linearly reductive quotient singularities are defined in terms of \'etale covers.
However, for stacky surfaces \( \cX \) we defined linearly reductive quotient singularities in terms of the existence of \emph{smooth} covers (since \'etale covers by schemes generally do not exist). We show that these two definitions are consistent, i.e.~if \(X \) is a surface and there is a smooth cover of the form \(U/H\) with \(U\) smooth and \(H\) finite linearly reductive, then there is an \'etale cover of this form.

If the characteristic of \( \bfk \) is at least 7, then for schemes of dimension \( 2 \), \(F\)-regular singularities are the same as linearly reductive quotient singularities by~\cite[Theorem~11.5]{2021arXiv210201067L}. Also, \(U\) is \(F\)-regular since it is smooth over a perfect field and since direct summands inherit \(F\)-regularity, \(U/H\) is strongly \(F\)-regular. Next, since smooth morphisms \'etale locally have sections, we may assume we have a smooth surjection \(U/H \to X \) with a section. The existence of such a section implies that \(\cO_X \) is a direct summand of the structure sheaf of \(U/H\), so \(X \) is strongly \(F\)-regular, hence has linearly reductive quotient singularities.

In characteristic \( 0 \), we can similarly argue by replacing ``\(F\)-regular'' with ``log terminal'' to conclude that, again in dimension 2, being a quotient singularity is a local property in smooth topology (see, say,~\cite[Proposition~2.15]{MR3057950}).

Unfortunately it is not clear to the authors if~\cite[Theorem~11.5]{2021arXiv210201067L} generalizes to singularities over perfect fields. Actually, this prevents them from applying the main arguments of this paper to not-necessarily algebraically closed fields. The issue lies in that the classification of canonical surface singularities is known only over algebraically closed fields.

For a general perfect field
\(
   \bfk
\),
we can still find a finite Galois extension
\(
   \bfk
   \subset
   L
\)
and present the base change of the original singularity to \(L\) as a quotient by a linearly reductive group scheme (over \(L\)). As we illustrate in~\cref{example:non-closed field}, in good cases the Galois group is linearly reductive and we can combine it with (a descent to \( \bfk \) of) the group scheme, so as to prove that the original singularity over \( \bfk \) is also a linearly reductive quotient singularity.
However, it is unclear if such an argument always works or not.
\end{remark}

\begin{example}\label{example:non-closed field}
    Consider the following singularity.
    \begin{align}
        \bfk
       & =
        \bF _{ 3 }\\
        R
       & =
        \bfk [ x, y, z ]
        /
        (
            z ^{ 3 } - x ^{ 2 } - y ^{ 2 }
        )
    \end{align}

    Note that
    \(
       R
    \)
    is not isomorphic to the \(A _{ 2 }\)-singularity unless we extend the base field
    \(
       \bfk
    \).
    On the other hand, in doing so we can find a presentation of
    \(
       R
    \)
    as a linearly reductive quotient singularity as follows.
    
    Put
    \begin{align}
        L
        =
        \bfk [ \sqrt{ - 1 } ]
        =
        \bfk [ i ] / ( i ^{ 2 } + 1 ),
    \end{align}
    so that
    \begin{align}
        R _{ L }
        =
        R \otimes _{ \bfk } L
        \simeq
        L [ x, y, z ]
        /
        \left(
            z ^{ 3 } - ( x + i y ) ( x - i y )
        \right).
    \end{align}
    It is an
    \(
       L
    \)-subalgebra of
    \(
       S
       =
       L [ u, v ]
    \)
    by the following map.
    \begin{align}
        x + i y
        &
        \mapsto
        u ^{ 3 }\\
        x - i y
        &
        \mapsto
        v ^{ 3 }\\
        z
        &
        \mapsto
        u v
    \end{align}
    The
    \(
       L
    \)-algebra
    \(
       S
    \)
    admits a
    \(
       \bZ / 3 \bZ
    \)-grading given by
    \begin{align}
        \deg u = \overline{1},
        \deg v = \overline{- 1}.
    \end{align}
    It corresponds to a
    \(
       \mu _{ 3 }
    \)-action on
    \(
       S
    \)
    and the homomorphism above identifies
    \(
       R _{ L }
    \)
    with the invariant subring
    \(
       S ^{ \mu _{ 3 } }
    \).

    On the other hand, the
    \(
       \bfk
    \)-algebra
    \begin{align}
        S
        \simeq
        \bfk [ u, v, i ]
        /
        (
            i ^{ 2 } + 1
        )
    \end{align}
    admits a
    \(
       \mu _{ 6 }
    \)-action corresponding to the
    \(
       \bZ / 6 \bZ
    \)-grading as follows.
    \begin{align}
        \deg u
        &
        =
        \overline{2}\\
        \deg v
        &
        =
        \overline{- 2}\\
        \deg i
        &
        =
        \overline{3}
    \end{align}

    We have the following short exact sequence of group schemes over
    \(
       \bfk
    \).
    \begin{align}
        1
        \to
        \mu _{ 3 }
        \to
        \mu _{ 6 }
        \to
        \Aut _{ \bfk } ( L )
        \to
        1
    \end{align}
    
    Indeed, the Galois group
    \(
        \Aut _{ \bfk } ( L )
        \simeq
        \bZ / 2 \bZ
    \)
    is generated by the automorphism
    \begin{align}
        i
        \mapsto
        - i
    \end{align}
    of
    \(
       L
    \)
    and it lifts to the action
    \(
       \mu _{ 6 }
       \curvearrowright
       S
    \)
    via
    \(
       \left(
           \mu _{ 6 }
       \right) _{ \mathrm{red} }
       \simto
       \Aut _{ \bfk } ( L )
    \).
    The map
    \(
        \mu _{ 3 }
        \to
        \mu _{ 6 }
    \)
    is the dual of the standard quotient map
    \(
       \bZ / 6 \bZ
       \twoheadrightarrow
       \bZ / 3 \bZ       
    \).
    The induced action
    \(
       \mu _{ 3 }
       \curvearrowright
       S
    \)
    coincides with the one above.
    Hence
    \(
       S ^{ \mu _{ 6 } }
       \simeq
       R
    \)
    and
    \(
       R
    \)
    is a linearly reductive quotient singularity over
    \(
       \bfk
    \).
\end{example}

Let us resume the analysis of
\(
   \Lambda _{ 1 }
\).

\begin{proposition}\label{proposition:R1 is SH}
There is a finite linearly reductive group scheme \( H \) and a linear action on
\(
    S \coloneqq \bfk [[ x, y ]]
\)
which is free on \( \Spec S \setminus \{ 0 \} \) such that 
\(
    R _{ 1 } = S ^{ H }
\).
\end{proposition}

\begin{proof}
    By~\cite[Proposition~5.1]{MR978602}, \( \Lambda _{ 1 } \) is reflexive Azumaya and tame. In particular, it follows that
    \begin{align}\label{equation:Lambda1 as reflexive pull back of Lambda}
        \Lambda _{ 1 } \simeq \left( \Lambda \otimes _{ R } R _{ 1 } \right) ^{ \vee \vee }.
    \end{align}
    Since we assume in~\cref{situation:local} that \( R \) is a direct summand of \( \Lambda \), it follows from~\eqref{equation:Lambda1 as reflexive pull back of Lambda} that \( R _{ 1 } \) is a direct summand of \( \Lambda _{ 1 } \). Hence it is of finite representation type. In particular, the divisor class group of \( R _{ 1 } \) is finite and hence \( R _{ 1 } \) is \( \bQ \)-Gorenstein. We let
    \(
        \Spec \Delta
        \to
        \Spec R _{ 1 }
    \)
    denote an index 1 cover of \( \Spec R _{ 1 } \), so that \( \Delta \) is \( \bZ / r \)-graded for some \( r > 0 \) and
    \(
        R _{ 1 }
        =
        \Delta
        _{ \overline{ 0 } }
    \)
    (\( \Delta \) is nothing but the algebra \( \Lambda _{ 1 } \) as defined in~\eqref{equation:definition of Lambda1} for \( \Lambda = R _{ 1 } \)). It follows that the divisor class group of \( \Spec \Delta \) is also finite. Hence by~\cite[Proposition~(17.3)]{MR276239}, it is a rational singularity. Since \( \Spec \Delta \) is also Gorenstein, it follows that \( \Spec \Delta \) is a canonical singularity. Therefore, since we assume that the characteristic is \( 0 \) or \( p \ge 7 \),
    \cite[Proposition~4.2]{MR3161094} implies that
    \(
        \Spec \Delta
    \)
    is a linearly reductive quotient singularity.

    Thus we have confirmed that \( R _{ 1 } \) is a direct summand of the ring of formal power series in \( 2 \)-variables, so that it is strongly \( F \)-regular. Hence \( R _{ 1 } \) is a linearly reductive quotient singularity by~\cite[Theorem~11.5~(3)\(\Rightarrow\)(2)]{2021arXiv210201067L}. Let \( N \triangleleft H \) be the subgroup scheme generated by pseudo-reflections, see~\cite[Definition 1.2]{MR2950159}. Replacing \( H \) with \( H / N \) and \( \bA ^{ 2 } \) with
    \(
        \bA ^{ 2 } / N \simeq \bA ^{ 2 },
    \)
    by Proposition 2.10 of (loc.~cit.), we may assume \(H\) acts on \(\bA ^{ 2 }\) without pseudo-reflections. Then Theorem 1.9 of (loc.~cit.) shows that the action of \( H \) on \( \Spec S \) is free outside the origin and \( R _{ 1 } \simeq S ^{ H } \).
\end{proof}
\begin{remark}
  \cref{proposition:R1 is SH} is the only proof where we use the assumption that \( R \) is a direct summand.
\end{remark}

Take any
\( H \curvearrowright \Spec S \)
as in~\cref{proposition:R1 is SH}.
Put
\begin{align}
    \Gamma
    \coloneqq
    \Lambda _{ 1 } \overline{ \otimes } _{ R _{ 1 } } S
    \stackrel{\text{\cite[p.~4]{MR978602}}}{\coloneqq} \left( \Lambda _{ 1 } \otimes _{ R _{ 1 } } S \right) ^{ \ast \ast }.
\end{align}

The algebra \( \Gamma \) inherits the left \( H \)-action on \( \Spec S \); i.e., it is a right comodule algebra over the Hopf algebra \( \cO _{ H } \) by the coaction map
\begin{align}
    \rho \colon \Gamma \to \Gamma \otimes _{ \bfk } \cO _{ H }.
\end{align}
Dually, \( \Gamma \) is a left module algebra over the dual Hopf algebra \( \cO _{ H } ^{ \vee } \) by the following action map, which corresponds to \( \rho \) under the obvious adjunction.
\begin{align}
    \cO _{ H } ^{ \vee } \otimes _{ \bfk } \Gamma \to \Gamma
\end{align}

The relationship among the relevant algebras is summarized in the following commutative diagram.
\begin{figure}[H]\label{equation:the local diagram}
    \centering
    \includegraphics[scale=1.2]{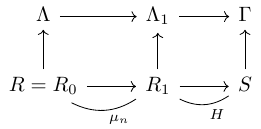}
\end{figure}



\begin{proposition}\label{proposition:Gamma * H simeq End Gamma}
    There is an isomorphism of algebras as follows.
    \begin{align}
        \Gamma \ast H \simeq \End _{ \Lambda _{ 1 } } \left( \Gamma \right)
    \end{align}
\end{proposition}

\begin{proof}
    By~\cite[THEOREM~8.3.3~1)\(\Rightarrow\)2)]{Montgomery}, it is enough to show that
    \begin{align}\label{equation:Lambda1 as invariant subring}
        \Lambda _{ 1 } = \Gamma ^{ H } \coloneqq \Gamma ^{ \cO _{ H } ^{ \vee } } = \Gamma ^{ \co \cO _{ H } },
    \end{align}
    where the last term denotes the subalgebra of coinvariants (\cite[Definition~1.7.1~1)]{Montgomery}),
    and that
    \(
        \Lambda _{ 1 } \to \Gamma
    \)
    is a right \( \cO _{ H } \)-Galois extension in the sense of~\cite[Definition~8.1.1]{Montgomery}; i.e., the following map is an isomorphism.

    \begin{align}\label{equation:Gamma is H-Galois over Lambda1}
        \beta \colon \Gamma \otimes _{ \Lambda _{ 1 } } \Gamma \to \Gamma \otimes _{ \bfk } \cO _{ H }; \quad a \otimes b \mapsto ( a \otimes 1 ) \rho ( b )
    \end{align}

    Note that both \( \Gamma \) and \( \Lambda _{ 1 } \) are reflexive \( R _{ 1 } \)-modules. So is \( \Gamma ^{ \cO _{ H } ^{ \vee } } \), as it is a direct summand of \( \Gamma \) as an \( R _{ 1 } \)-module for the reason that \( H \) is linearly reductive. Therefore the equality~\eqref{equation:Lambda1 as invariant subring} can be checked in codimension one over \( R _{ 1 } \).

    Similarly, note that \( \Gamma \) is a free module over \( S \), hence so is the right hand side of~\eqref{equation:Gamma is H-Galois over Lambda1}. On the other hand, since \( \Lambda _{ 1 } \) has right global dimension \( 2 \) by~\cref{corollary:Lambda1 has global dimension 2} and \( \Gamma \) is a reflexive \( R _{ 1 } \)-module by construction, it is a projective right \( \Lambda _{ 1 } \)-module by~\cite[Theorem~1.10]{MR245572} or~\cite[Proposition~2.17~(2)\(\Rightarrow\)(3)]{MR3251829}. Hence the left hand side of~\eqref{equation:Gamma is H-Galois over Lambda1} is a projective right \( \Gamma\)-module. In particular, it is a free \( S \)-module. Therefore it is enough to show \( \beta \) is an isomorphism in codimension \( 1 \) over \( S \).

    As \( R _{ 1 } = S ^{ H } \) and the action \( H \curvearrowright \Spec S \) is free outside the origin, both assertions hold in codimension \( 1 \). To see this, let \( X, Y \) be the punctured spectra of \( S \) and \( R _{ 1 } \), respectively, so that the natural map \( \pi \colon X \to Y \) is isomorphic to the quotient map \( X \to \left[ X / H \right] \). For an \( S \)-module \( M \), put \( M ^{ \circ } \coloneqq M \vert _{ X } \) and so on (similarly for \( R _{ 1 } \)-modules). By construction, it then follows that \( \pi ^{ \ast } \Lambda _{ 1 } ^{ \circ } = \Lambda _{ 1 } ^{ \circ } \otimes _{ \cO _{ Y } } \cO _{ X } \simeq \Gamma ^{ \circ } \) and
    \(
        \left( \pi _{ \ast } \Gamma ^{ \circ } \right) ^{ H } \simeq \Lambda _{ 1 } ^{ \circ }
    \), hence~\eqref{equation:Lambda1 as invariant subring} holds in codimension \( 1 \) over \( \Spec R _{ 1 } \).
    On the other hand, it follows from
    \(
        Y \simeq \left[ X / H \right]
    \)
    that the natural map
    \begin{align}
        \cO _{ X } \otimes _{ \cO _{ Y } } \cO _{ X } \to \cO _{ X } \otimes _{ \bfk } \cO _{ H }
    \end{align}
    is an isomorphism. Applying
    \(
        \Lambda _{ 1 } ^{ \circ } \otimes _{ \cO _{ Y } } -
    \)
    to this, we see that~\eqref{equation:Gamma is H-Galois over Lambda1} is an isomorphism in codimension \( 1 \) over \( \Spec S \).
\end{proof}

\begin{corollary}\label{corollary:Gamma vs Lambda1}
    There is an equivalence of categories as follows.
    \begin{align}
        \module ^{ H } \Gamma \simeq \module \Lambda _{ 1 }
    \end{align}
\end{corollary}

\begin{proof}
    The proof is parallel to that of~\cref{corollary:Lambda1 vs Lambda}, where~\cref{proposition:Lambda1 * mun simeq End Lambda1} is replaced with~\cref{proposition:Gamma * H simeq End Gamma}. The argument similar to the last three lines of the proof of~\cref{corollary:Lambda1 vs Lambda} works here as well, as \( \Lambda _{ 1 } \) has global dimension \( 2 \) by~\cref{corollary:Lambda1 has global dimension 2}.
\end{proof}

%
%
\section{The root stack \( (\cX _{ r }, \cA _{ r }) \)}\label{section:root stack}

Let \( ( X, \cA )\) be as in~\cref{definition:nc-surface}. In this section we construct an algebraic stack
\(
    r \colon \cX _{ r } \to X
\)
over \( X \) and an order \( \cA _{ r } \) on \( \cX _{ r } \), which is Azumaya in codimension \( 1 \). Furthermore, this stack is \emph{tame} in the sense of~\cite{MR2427954}. Tame stacks are algebraic stacks whose stabilizers are finite linearly reductive group schemes, see~\cite[Theorem 3.2]{MR2427954}; note that this includes examples with \(\mu_p\)-stabilizers in characteristic \(p\) and hence tame stacks are more general than tame Deligne--Mumford stacks (i.e.~Deligne--Mumford stacks whose stabilizers have order prime to the characteristic). Although tame stacks are Artin stacks in general, they enjoy many of the same properties of tame Deligne--Mumford stacks, e.g.~they have a coarse space whose formation commutes with \emph{arbitrary} base change, and the push-forward functor to the coarse space is exact.

Consider the \( \bZ \)-graded \( \cA \)-algebra \( \cAtilde \) which is defined as follows, where
\(
    \omega _{ \cA }
    =
    \cHom _{ X } \left( \cA, \omega _{ X }, \right)
\)
is the dualizing bimodule of \( \cA \) (see~\eqref{equation:dualizing bimodule of cA} for the definition) and
\begin{align}
    \omega _{ \cA  } ^{ - 1 } \coloneqq \cHom _{ \cA } \left( \omega _{ \cA }, \cA \right)
\end{align}
is the dual \( \cA \)-bimodule of \( \omega _{ \cA } \).
The symbol \( { ( - ) } ^{ ( j ) }\) denotes the reflexive hull of \( ( - ) ^{ \otimes j } \) (see~\cref{definition:Rees algebra construction} for the local version).
\begin{align}
    \cAtilde
    \coloneqq
    \cA [ \omega _{ \cA } ^{ - 1 } ]
    =
    \bigoplus _{ j \in \bZ }
    \left(
        \omega _{ \cA } ^{ - 1 }
    \right)
    ^{ ( j ) }
\end{align}

The centre \( \zentrum ( \cAtilde ) \) is a \( \bZ \)-graded \( \cO _{ X } \)-algebra and hence
\begin{align}\label{eqn:Xtilde}
    \Xtilde \coloneqq \Spec _{ X } \zentrum ( \cAtilde )
\end{align}
is equipped with a \( \bG _{ m } \)-action. Put
\begin{align}\label{equation:cXr}
    \cX _{ r } \coloneqq \left[ \Xtilde / \bG _{ m } \right].
\end{align}
Note that the natural morphism
\begin{align}
    r \colon \cX _{ r } \to X
\end{align}
is the coarse moduli map.

\begin{lemma}\label{lemma:Xtilde is of finite type over X}
    \( \zentrum ( \cAtilde )\) is of finite type over \( \cO _{ X }\) and \( \cAtilde \) is coherent as an \( \cO _{ \Xtilde } \)-module.
\end{lemma}

\begin{proof}
    Since the ramification index at each prime divides \(n\), the rank of \( \cA \) by~\cref{remark:dualizingStuff}, we see that at each point
    \(
       x
       \in
       X
    \)
    the completion of
    \(
        \left( \omega _{ \cA } ^{ - 1 } \right) ^{ ( n ) }
    \)
    is isomorphic to the completion
    \(
       \cAhat
       \coloneqq
       \cA
       \otimes
       \widehat{
        \mathcal{ O }
        }
        _{
            X,
            x
        }
    \)
    of
    \(
       \mathcal{ A }
    \),
    as
    \(
        \cAhat
    \)-bimodules. Therefore the completion of the following sheaf of homomorphisms of \( \cA \)-bimodules
    \begin{align}
        \cL
        \coloneqq
        \cHom _{ \cA ^{ e } }
        \left(
            \cA, \left( \omega _{ \cA } ^{ - 1 } \right) ^{ ( n ) }
        \right)
    \end{align}
    satisfies
    \begin{align}
        \cL
        \otimes
        \widehat{
            \mathcal{ O }
        }
        _{
            X,
            x
        }
        \simeq        
        \cHom
        _{
            \cAhat ^{ e }
        }
        \left(
            \cAhat,
            \cAhat
        \right)
        \simeq
        \zentrum
        \left(
            \cAhat
        \right)
        =
        \widehat{
            \mathcal{ O }
        }
        _{
            X,
            x
        }.
    \end{align}
    Thus we conclude that \( \cL \) is an invertible sheaf on \( X \) and that the natural map
    \(
        \cL \otimes _{ \cO _{ X } } \cA \to \left( \omega _{ \cA } ^{ - 1 } \right) ^{ ( n ) }
    \)
    is an isomorphism of \( \cA \)-bimodules.

    Take a Zariski open subset \( U \subseteq X \) over which \( \cL \) is trivial, and fix an isomorphism of \( \cA | _{ U } \)-bimodules
        \(
            \varphi \colon \cA | _{ U }
            \simto
            \left( \omega _{ \cA } ^{ - 1 } \right) ^{ ( n ) } \vert _{ U }
        \).
    Put \( t \coloneqq \varphi ( 1 ) \). Then
    \begin{align}
        \cO _{ U } [ t, t ^{ - 1 } ]
        \subseteq
        \zentrum
        \left(
            \left(
                \cAtilde \vert _{ U }
            \right)
            ^{
                ( n )
            }
        \right),
    \end{align}
    where
    \(
        \left(
            \cAtilde \vert _{ U }
        \right)
        ^{
            ( n )
        }
    \)
    is the
    \(
       n
    \)-th Veronese subalgebra of
    \(
        \cAtilde \vert _{ U }
    \).
    This implies that the algebra
    \(
        \cAtilde \vert _{ U }
    \)
    is generated by the
    \(
       \cO _{ U }
    \)-coherent sheaf
    \(
       \bigoplus
       _{
           j = 0
       }
       ^{
           n - 1
       }
       \cA ^{ ( j ) } \vert _{ U }
    \)
    as an algebra over
    \(
        \cO _{ U } [ t, t ^{ - 1 } ]
    \).
    This implies that    
    \(
       \cAtilde \vert _{ U }
    \)
    is finitely generated as a module over
    \(
        \cO _{ U } [ t, t ^{ - 1 } ]
    \)
    and hence
    \(
       \zentrum
       \left(
        \cAtilde \vert _{ U }
       \right)
    \)
    is finitely generated as an algebra over
    \(
       \cO _{ U }
    \).
\end{proof}

Let \( \cA _{ r } \) be the \( \cO _{ \cX _{ r } } \)-algebra corresponding to \( \cAtilde \) under the following equivalence of monoidal categories.
\begin{align}
    \coh \cX _{ r } \simeq \coh ^{ \bG _{ m } } \Xtilde
\end{align}

Let
\(
    U = \Spec R \to X
\)
be the standard morphism from the completion \( R \) of the local ring of a closed point of \( X \).
Put
\(
    \Lambda = \cA \otimes _{ \cO _{ X } } R
\).

\begin{lemma}
    The pair \( ( R, \Lambda )\) satisfies the assumptions of~\cref{situation:local}.
\end{lemma}
\begin{proof}
    By the assumptions of~\cref{definition:nc-surface}, \( \cA \) is of global dimension \( 2 \) and tame. This implies that
    \(
        \Lambda = \cA \otimes _{ \cO _{ X } } R
    \)
    has the same properties.

    It also follows that \( \zentrum ( \Lambda ) = R \), since in general the centre of the completion of an algebra coincides with the completion of the centre. This follows from the general fact that if \( \Lambda \) is an algebra with \( \zentrum ( \Lambda ) = R \) and \( R \to S \) is a flat morphism of commutative algebras, then there is a sequence of isomorphisms as follows.
    \begin{align}
        \zentrum ( \Lambda \otimes _{ R } S )
        \simeq
        \HH ^{ 0 } \left( \Lambda \otimes _{ R } S / \bfk \right)
        \simeq
        \HH ^{ 0 } \left( \Lambda \otimes _{ R } S / S \right)\\
        \stackrel{\text{\(R\) is flat over \( S \)}}{\simeq}
        \HH ^{ 0 } \left( \Lambda / R \right) \otimes _{ R } S
        \simeq
        \HH ^{ 0 } \left( \Lambda / \bfk \right) \otimes _{ R } S
        \simeq
        \zentrum ( \Lambda ) \otimes _{ R } S.
    \end{align}
\end{proof}

The next result relates the local objects introduced in~\cref{section:Reiten-Van den Bergh} to the global objects which have we just constructed.

\begin{proposition}
\begin{align}\label{equation:global vs local}
    \Xtilde \times _{ X } U
    &
    \simeq
    \Spec \Rtilde \quad \text{(\(\bGm\)-equivariant)}
    \\
    \cX _{ r } \times _{ X } U
    &
    \simeq
    \left[ \Spec \Rtilde / \bGm \right]
    \simeq
    \left[ \Spec R _{ 1 } / \mu _{ n } \right]
    \\
    \cA _{ r } \vert _{ U }
    &
    \simeq
    \Lambdatilde
\end{align}
\end{proposition}

\begin{proof}
    For each \( i \in \bZ \), we have
    \begin{align}
        \left( \omega _{ \cA } ^{ - 1 } \right) ^{ ( i ) } \otimes _{ \cO _{ X } } R
        &
        \simeq
        \cHom _{ X } \left( \cHom _{ X } \left( \left( \omega _{ \cA } ^{ - 1 } \right) ^{ \otimes i }, \cO _{ X } \right), \cO _{ X } \right)
        \otimes _{ \cO _{ X } } R
        \\
        &
        \simeq
        \Hom _{ R } \left( \Hom _{ R } \left( \left( \omega _{ \cA } ^{ - 1 } \otimes _{ \cO _{ X } } R \right) ^{ \otimes i }, R \right), R \right)
        \simeq
        \left(\omega _{ \Lambda } ^{ - 1 }\right) ^{ ( i ) }.
    \end{align}

    By~\cite[Lemma~2.5~(3)]{MR2492474} or~\cref{remark:bimodule}, the dualizing bimodule localizes and so
    \(
        \omega_{ \cA } ^{ - 1 } \otimes _{ \cO _{ X } } R \simeq \omega _{ \Lambda } ^{ - 1 }
    \).
    Thus we obtain the following isomorphism of \( \bZ \)-graded \( R \)-algebras
    \begin{align}
        \cAtilde \otimes _{ \cO _{ X } } R
        \simeq
        \Lambdatilde
    \end{align}
    Hence
    \begin{align}
        \zentrum ( \cAtilde ) \otimes _{ \cO _{ X } } R
        \simeq
        \zentrum ( \cAtilde ) \otimes _{ \cO _{ \Xtilde } } \Rtilde
        \stackrel{\text{\cref{lemma:Xtilde is of finite type over X}}}{\simeq}
        \zentrum ( \cAtilde \otimes _{ \cO _{ \Xtilde } } \Rtilde )
        \simeq
        \zentrum \left( \cAtilde \otimes _{ \cO _{ X } } R \right)
        \simeq
        \zentrum ( \Lambdatilde )
        =
        \Rtilde,
    \end{align}
    so that we obtain the following \( \bGm \)-equivariant isomorphism.
    \begin{align}
        \Xtilde \times _{ X } U \simeq \Spec \Rtilde
    \end{align}
    The rest of the assertions immediately follow from this.
\end{proof}

\begin{proposition}\label{proposition:tame_algebraic_lrqs}
    The stack
    \( \cX _{ r } \)
    is a normal tame algebraic stack with only linearly reductive quotient singularities.
\end{proposition}

\begin{proof}
We already know that \(\cX_r\) is an algebraic stack.
It follows from~\eqref{equation:global vs local} and~\cite[Theorem~3.2]{MR2427954} that it is a tame algebraic stack. Note that, although~\cite[Theorem~3.2~(c)]{MR2427954} is stated only for fppf covers, it follows from~\cite[Corollary~3.4]{MR2427954} and~\cite[Corollary~3.5]{MR2427954} that if \( f \colon Y \to X \) is any surjective morphism of schemes and the base change of a stack over \( X \) by \( f \) is tame, then so is the original stack over \( X \).
Finally, it follows from~\eqref{equation:global vs local} and the results in~\cref{section:Reiten-Van den Bergh} that \( \cX _{ r } \) is normal and has only linearly reductive quotient singularities.
\end{proof}

%
%
\section{The canonical stack \( \left( \cX _{ c }, \cA _{ c } \right) \)}

%
%
\subsection{Review of canonical stacks}
Suppose \( \bfk \) is a field and $X$ is a \( \bfk \)-variety with tame quotient singularities, i.e.~there is an \'etale cover $\{X_i\to X\} _{ i }$ such that each $X_i$ is a quotient $U_i/G_i$ with $U_i$ smooth and $G_i$ a finite group whose order is prime to the characteristic of \( \bfk \). In \cite[Proof of Proposition~2.8 and~(2.9)]{MR1005008}, Vistoli constructed a canonical smooth tame Deligne--Mumford stack $\cX$ with coarse space $\pi\colon\cX\to X$. Moreover, $\pi$ is an isomorphism over the smooth locus $X^{\textrm{sm}}$ and hence $\cX$ only adds stacky structure at the singular points of $X$. This \emph{canonical stack} has a universal property, see e.g., \cite[Theorem~4.6]{MR2774310}:~it is terminal for dominant codimension-preserving morphisms $\mathcal{Z}\to X$, where $\mathcal{Z}$ is a smooth tame Deligne--Mumford stack; their theorem is stated in characteristic $0$ but the proof works equally well in arbitrary characteristic when one imposes tameness assumptions. The construction of the canonical stack $\cX\to X$ proceeds as follows:~due to the universal property, it suffices to construct the stack locally and hence one may assume $X=U/G$ with $U$ smooth and $G$ a finite group whose order is prime to the characteristic of \( \bfk \). In this case, one makes use of the Chevalley--Shephard--Todd Theorem \cite{MR0072877} which allows one to assume $G$ acts on $U$ without pseudo-reflections. As a result, the coarse space map $[U/G]\to X$ is an isomorphism over the smooth locus, thereby showing the existence of the canonical stack in the local setting.
For our work, we must consider \( \bfk \)-varieties $X$ with singularities that are worse than tame quotient singularities. For example, the $A_1$-singularity $\Spec \bfk[x,y,z]/(xz-y^2)$ is not of the form $U/G$ with $U$ smooth and $G$ a finite group (even if one allows the characteristic of the field to divide the order of the group $G$). Instead, it is $\mathbb{A}^2/\mu_2$; note that for fields of characteristic prime to $2$, $\mu_2\simeq\mathbb{Z}/2$. However, when \( \bfk \) has characteristic $2$, $\mu_2$ is a finite group scheme which is non-reduced. Despite this, $\mu_2$ is in many ways better behaved than $\mathbb{Z}/2$ in characteristic $2$, e.g., it is linearly reductive whereas $\mathbb{Z}/2$ has representations that are not split.
We say $X$ has linearly reductive quotient singularities if there is an \'etale cover $\{X_i\to X\}$ such that each $X_i$ is a quotient $U_i/G_i$ with $U_i$ smooth and $G_i$ a finite linearly reductive group scheme. For such $X$, the fourth author generalized Vistoli's construction, giving a canonical smooth tame algebraic stack $\cX$ with coarse space $\pi\colon\cX\to X$ which is an isomorphism over $X^{\textrm{sm}}$, see~\cite[Theorem~1.10]{MR2950159}. We refer to $\cX$ as the \emph{canonical stack} of $X$. As in Vistoli's construction, $\cX$ is constructed \'etale locally on $X$ and glued using a universal property; note that the Chevalley--Shephard--Todd theorem does not apply to finite linearly reductive group \emph{schemes} and so the main input to the canonical stack construction involves proving an appropriate generalization of the the Chevalley--Shephard--Todd theorem. In particular, $\cX$ is also given locally by $[U/G]$ where $U$ is a smooth and $G$ is a finite linearly reductive group scheme acting without pseudo-refections, see \cite[Definition 1.2]{MR2950159}.
\begin{remark}\label{remark:can-stack-over-stack}
Although \cite[Theorem~1.10]{MR2950159} is only stated for the case where the base is a scheme, the proof works equally well when the base is a stack, as we now explain. Let \(M\) be a tame algebraic stack with only linearly reductive quotient singularities, which by definition means there is a smooth cover \(\{M_i\to M\}\) with \(M_i=U_i/G_i\), \(U_i\) a smooth scheme, and \(G_i\) a finite linearly reductive group scheme. Let \(\cX_i\to M_i\) be the canonical tame stack given by Theorem 1.10 of (loc.~cit.). Letting \(M_{ij}=M_i\times_M M_j\), Lemma~5.5 of (loc.~cit.) yields an isomorphism between \(\cX_i\times_{M_i}M_{ij}\) and \(\cX_j\times_{M_i}M_{ij}\) which is unique up to 2-isomorphism; this is because canonical stacks commute with smooth base change.
Thus, we obtain our canonical stack \(\cX\to M\) with the property that \(\cX_i=\cX\times_{M}M_i\).
\end{remark}

Since~\cref{proposition:tame_algebraic_lrqs}, shows that \( \cX _{ r } \) is tame and algebraic with linearly reductive quotient singularities, we have an associated canonical stack
\begin{align}\label{equation:canonical to root}
    c \colon \cX _{ c } \to \cX _{ r },
\end{align}The stack \(\cX _{ c }\) is a smooth tame algebraic stack, and \(c\) is a relative coarse space map which is an isomorphism over the smooth locus of \( \cX _{ r } \).

%
%
\subsection{Morita equivalences}

We first prove a general Morita-equivalence criterion.
Let $c\colon\cX\to\cY$ be a morphism of finite type algebraic stacks over \( \bfk \).
Assume that
\(
   c _{ \ast }
   \colon
   \Qcoh \cX
   \to
   \Qcoh \cY
\)
is an exact functor. Let $\cA_\cX$ (resp.~$\cA_\cY$) be a coherent sheaf of $\cO _{
    \cX
}$-algebras (resp.~$\cO _{ \cY }$-algebras) and suppose we have an isomorphism
$
    \alpha\colon \cA_\cY\to c_{ \ast }\cA_\cX
$
of
\(
   \cO
   _{
    \cY
   }
\)-algebras. Let
\(
   \alpha '
   \colon
   c
   ^{
    \ast
   }
   \cA
   _{
    \cY
   }
   \to
   \cA
   _{
    \cX
   }
\)
be the morphism corresponding to
\(
   \alpha
\)
under the adjoint pair \( c^{ \ast } \dashv c_{ \ast } \).
In this situation we obtain the pair of adjoint functors
    \begin{align}\label{equation:adjoint pair (G,F)}
        G \colon \coh \cA_\cY \rightleftarrows \coh \cA _\cX \colon F,
    \end{align}
where
\begin{align}
    F
    \left(
        \cM
    \right)
    &
    =
    c _{ \ast }
    \cM,
    \\
    G(\cN)
    &
    =
    c^{ \ast }\cN \otimes_{ c^{ \ast }\cA_\cY, \alpha '} \cA _\cX.
\end{align}
For
\(
   \cM
   \in
   \coh
   \cA
   _{
    \cX
   }
\),
let
\begin{align}\label{equation:adjunction counit}
    \eta = \eta _{ \cM } \colon G F ( \cM ) = c^{ \ast }c_{ \ast }\cM \otimes _{ c^{ \ast }\cA_\cY } \cA _\cX \to \cM.
\end{align}
be the adjunction counit map evaluated at
\(
   \cM
\).
\begin{theorem}\label{theorem:general equivalence}
    Suppose that the map~\eqref{equation:adjunction counit} is an isomorphism for all $\cM \in \coh \cA_\cX$. Then the functor
    \(
       G
    \)
    is an equivalence of \( \bfk \)-linear categories.
\end{theorem}

\begin{proof}
We first show \( G \) is fully faithful; equivalently, we show that the adjunction unit
\(
    \mu \colon \id \to F G
\)
is a natural isomorphism.
For each \( \cN \in \coh \cA_\cY \), the map
\(
    \mu _{ \cN } \colon \cN \to F ( G ( \cN ) )
\)
coincides with the composition of the following morphisms
\[
    \cN \xrightarrow[\text{\(\alpha\) is iso}]{\sim}
    \cN \otimes_{\cA, \alpha} c_{ \ast } \cA _\cX \mathrel{\mathop{\longrightarrow}}
   c_{ \ast }( c^{ \ast } \cN \otimes_{ c^{ \ast }\cA_\cY} \cA _\cX )
\]
so it is enough to confirm that the righthand morphism is an isomorphism. Since this map is defined globally, we need only prove that its restriction to any smooth cover \( U\to \cY\) is an isomorphism; we may look further locally and assume $U$ is affine. To ease notation, let \( \cN, \cA_\cY, \) etc.~denote their restriction to one of those affine opens. Then the morphism is obviously an isomorphism for \( \cN = \cA \). Both sides are right exact functors in \( \cN \), and the category of coherent modules over \( \cA_\cY \) is generated by \( \cA_\cY \) (over such an open affine subset). Thus we are done.

Finally, the essentially surjectivity of
\(
   G
\)
follows from the assumption.
\end{proof}

Our next aim is to apply~\cref{theorem:general equivalence} to $\cX$, $\cX_r$ and $\cX_c$. We do so after establishing some preliminary results. We put
\begin{align}\label{equation:cAc}
    \cA _{ c } \coloneqq \left( c ^{ \ast } \cA _{ r } \right) ^{ \vee \vee }.
\end{align}
Let \(R\) be the completion of a local ring of \(X \) and let 
\(
    U = \Spec R \to X
\)
be the induced map. Then \(R\) is as in~\cref{section:Reiten-Van den Bergh}. We freely use the notation introduced there.
Consider the base change of~\eqref{equation:canonical to root} by \( U \to X \):
\begin{align}\label{equation:from Xc to Xr local}
    c \vert _{ U } \colon
    \cX _{ c } \times _{ X } U \to \cX _{ r } \times _{ X } U
    \simeq
    \left[ \Spec R _{ 1 } / \mu _{ n } \right].
\end{align}
By~\cref{proposition:R1 is SH}, there is a linearly reductive finite subgroup scheme
\(
    H < \GL ( 2, \bfk )
\)
and a linear action
\(
    H \curvearrowright \Spec S
\),
where
\(
    S = \bfk [[ x, y ]]
\),
which is free outside the origin and
\(
    R _{ 1 } \simeq S ^{ H }
\).
The following commutative diagram of stacks describes the relationship between global objects and local objects,
where the symbol
\(
    \ulcorner   
\)
means that the square is Cartesian.
\begin{figure}[H]
    \centering
    \includegraphics[scale=1.2]{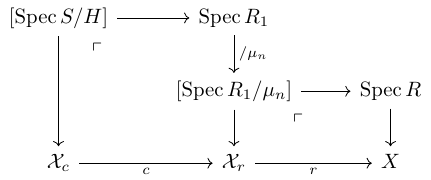}
    \caption{}
    \label{figure:local vs global}
\end{figure}

Since reflexive hulls commute with flat pullback, we see that
\begin{align}\label{equation:Gamma vs cAc}
    \Gamma \simeq \cA _{ c } \vert _{ \Spec S }.
\end{align}
Recall that \( \Gamma \) is equipped with an \( H \)-action. The smash product
\(
    \Gamma \ast H
\)
is identified with a sheaf of \( \cO _{ \left[ \Spec S / H \right] } \)-algebras, which will be denoted by the same symbol by abuse of notation. We have the following isomorphism.
\begin{align}\label{equation:Gamma ast H vs cA}
    M _{ n } ( S ) \ast H
    \simeq
    \Gamma \ast H \simeq \cA _{ c } \vert _{ \left[ \Spec S / H \right] }
\end{align}

\begin{proposition}
    \( \cA _{ c } \) is a sheaf of Azumaya algebras on the smooth tame algebraic stack \( \cX _{ c } \).
\end{proposition}

\begin{proof}
This follows from the isomorphism~\eqref{equation:Gamma vs cAc} and~\cref{remark:local criterion for Azumaya algebras}, since we know that \( \Gamma \) is an Azumaya algebra over \( S \) (see~\cref{section:Reiten-Van den Bergh}). Note that the morphism \( \Spec S \to \cX _{ c } \) is flat and jointly surjective if we vary \( \Spec R \to X \) over all points of \( X \).
\end{proof}

Let \( \pi = r \circ c \colon \cX _{ c } \to X \) be the coarse space morphism.
\begin{lemma}
    There is a canonical isomorphism of sheaves of \( \cO _{ X } \)-algebras as follows.
    \begin{align}
        \alpha
        \colon
        &
        \cA
        \to
        r _{ \ast } \cA _{ r } \label{equation:alpha}
        \\
        \beta \colon
        &
        \cA _{ r }
        \to
        c _{ \ast } \cA _{ c } \label{equation:beta}                
    \end{align}
\end{lemma}

\begin{proof}
We have the following canonical isomorphism
    \begin{align}
        \cA
        =
        \left(
            \text{the degree \( 0 \) part of } \cAtilde
        \right)
        \simeq r _{ \ast } \cA _{ r }
    \end{align}
    which yields~\eqref{equation:alpha}.

    As \( \cA _{ c } = ( c ^{ \ast } \cA _{ r } ) ^{ \vee \vee } \), there is a canonical morphism
    \(
        c _{ \ast } c ^{ \ast } \cA _{ r } \to c _{ \ast } \cA _{ c }
    \).
    The left hand side is isomorphic to
    \(
        \cA _{ r }
    \)
    by the projection formula and the fact that \( c _{ \ast } \cO _{ \cX _{ c } } \simeq \cO _{ \cX _{ r } } \). Thus we obtain the morphism~\eqref{equation:beta}.
    Note that \( c _{ \ast } \cA _{ c } \) is torsion free. In fact, otherwise there is a non-trivial subsheaf \( \cT \hookrightarrow c _{ \ast } \cA _{ c } \) whose support is strictly smaller than \( \cX _{ r } \). The corresponding morphism
    \(
        c ^{ \ast } \cT \to \cA _{ c }
    \)
    would be non-zero, but this is a contradiction since \( c ^{ \ast } \cT \) is torsion and \( \cA _{ c } \) is torsion free.
    Finally, since \( c \) is an isomorphism in codimension \( 1 \), \( \cA _{ r } \) is reflexive, and \( c _{ \ast } \cA _{ c } \) is torsion free, it follows that~\eqref{equation:beta} is an isomorphism.
\end{proof}

We are now ready to the prove the main theorem of this section.

\begin{theorem}\label{theorem:equivalence of cA and cAr and cAc}
    There are equivalences of \( \bfk \)-linear categories
    \begin{align}
        \coh \cA \simeq \coh \cA _{ r } \simeq \coh \cA _{ c }.
    \end{align}
\end{theorem}

\begin{proof}
We first consider the equivalence $\coh \cA \simeq \coh \cA _{ r }$. By~\cref{theorem:general equivalence}, it suffices to take an arbitrary object
\(
    \cM \in \coh \cA _{ r }
\)
and prove
\begin{align}
    \eta = \eta _{ \cM } \colon G F ( \cM ) = r ^{ \ast } r _{ \ast } \cM \otimes _{ r ^{ \ast } \cA } \cA _{ r } \to \cM.
\end{align}
is an isomorphism. By faithfully flat descent, we can and will check this locally on \( X \).
Take a formal neighborhood \( \Spec R \to X \) as before and consider the pullback \( \eta '\) of \( \eta \) to \( \left[ \Spec R _{ 1 } / \mu _{ n } \right] \)
(see~\cref{figure:local vs global} for the various stacks which appear in this proof). Since \( \Spec R \to X \) is flat, we can rewrite \( \eta '\) as follows. To ease notation, put
\(
    r ' \coloneqq r \vert _{ \left[ \Spec R _{ 1 } / \mu _{ n } \right] }
\),
\(
    \cA ' _{ r } \coloneqq \cA _{ r } \vert _{ \left[ \Spec R _{ 1 } / \mu _{ n } \right] }
\), and
\(
    \cM ' \coloneqq \cM \vert _{ \left[ \Spec R _{ 1 } / \mu _{ n } \right] }
\).
\begin{align}
    \eta ' \colon { r ' } ^{ \ast } { r ' } _{ \ast } \cM ' \otimes _{ { r ' } ^{ \ast } \Lambda } \cA ' _{ r } \to \cM '
\end{align}
This is a morphism in the category
\(
    \module \cA ' _{ r }
\).
Let
\(
    M \in \module ^{ \bZ / n } \Lambda _{ 1 }
\)
be the object corresponding to \( \cM ' \) under the equivalence of categories
\begin{align}
    \coh \cA ' _{ r } \simeq \module ^{ \bZ / n } \Lambda _{ 1 }.
\end{align}
Then \( \eta '\) is identified with the following morphism of the category
\(
    \module ^{ \bZ / n }
    \Lambda _{ 1 }
\), which is known to be an isomorphism.
\begin{align}
    M _{ \overline{ 0 } } \otimes _{ \Lambda } \Lambda _{ 1 } \to M
\end{align}
In fact, this is the adjunction counit map of the equivalence of categories~\cref{corollary:Lambda1 vs Lambda}.

Next, we turn to the equivalence $\coh \cA_{r} \simeq \coh \cA _{c}$. By~\cref{theorem:general equivalence}, it suffices to take an arbitrary object
\(
    \cM \in \coh \cA _{ c }
\)
and show
\begin{align}
    \eta = \eta _{ \cM } \colon G F ( \cM ) = c ^{ \ast } c _{ \ast } \cM \otimes _{ c ^{ \ast } \cA _{ r } } \cA _{ c } \to \cM
\end{align}
is an isomorphism. Again, by faithfully flat descent, we can and will check this locally on \( \cX _{ r } \).
Take a formal neighborhood \( \Spec R \to X \) as before and consider the pullback \( \eta '\) of \( \eta \) to \( \Spec R _{ 1 } \). Since \( \Spec R _{ 1 } \to X \) is flat, we can rewrite \( \eta '\) as follows. To ease notation, put
\(
    c ' \coloneqq c \vert _{ \left[ \Spec S / H \right] }
\),
\(
    \cA ' _{ c } \coloneqq \cA _{ c } \vert _{ \left[ \Spec S / H \right] }
\), and
\(
    \cM ' \coloneqq \cM \vert _{ \left[ \Spec S / H \right] }
\)
(see~\cref{figure:local vs global} for the various stacks which appear in this proof).
\begin{align}
    \eta ' \colon { c ' } ^{ \ast } { c ' } _{ \ast } \cM ' \otimes _{ { c ' } ^{ \ast } \Lambda _{ 1 } } \cA ' _{ c } \to \cM '
\end{align}
This is a morphism in the category
\(
    \coh \cA ' _{ c }
\).
On the other hand, let
\(
    M \in \module(\Gamma \ast H)
\)
be the object corresponding to \( \cM ' \) under the equivalence of categories
\begin{align}
    \coh \cA ' _{ c } \simeq \module(\Gamma \ast H),
\end{align}
which follows from the isomorphism of algebras~\eqref{equation:Gamma ast H vs cA}.
Then \( \eta '\) is identified with the following morphism of \( \Gamma \ast H \)-modules.
\begin{align}\label{equation:adjunction counit_local}
    \Hom _{ \Gamma \ast H } \left( \Gamma, M \right) \otimes _{ \Lambda _{ 1 } } \Gamma \to M
\end{align}
In fact, this follows from the commutative diagram below.
The functors \( F '\) and \( G ' \) in the diagram are obtained as the restriction of the adjoint pair~\eqref{equation:adjoint pair (G,F)} to the formal neighborhood:
\begin{align}
    G ' \colon \module \Lambda _{ 1 } \leftrightarrows \module \cA ' _{ c } \colon F '
\end{align}
\begin{figure}[H]
    \centering
    \includegraphics[scale=1.2]{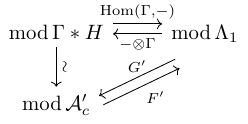}
\end{figure}
We already know that~\eqref{equation:adjunction counit_local} is an isomorphism, as it is the adjunction counit map of the equivalence of categories~\cref{corollary:Gamma vs Lambda1}.
\end{proof}

%
%
\section{The gerbe \( \cX _{ g } \)}
\label{section:gerbe}
In this section we associate a \( \mu _{ N } \)-gerbe
\begin{align}
    g \colon \cX _{ g } \to \cX _{ c }
\end{align}
to the pair
\(
    \left( \cX _{ c }, \cA _{ c } \right)
\)
of smooth tame (Artin) stacky surface and a sheaf of Azumaya algebras; we show that there is an equivalence of categories
\begin{align}\label{equation: coh cXg vs mod cAc}
    \coh ^{ ( 1 ) } \cX _{ g } \simeq \module \cA _{ c }
\end{align}
between the category of \emph{\(1\)-twisted coherent sheaves} on \( \cX _{ g }\) and the category of \( \cA _{ c } \)-modules. See~\cite[Definition~12.3.2]{MR3495343} for \(1\)-twisted coherent sheaves.

%
%
\subsection{Review of gerbes and twisted sheaves}\label{section:Review of gerbes and twisted sheaves}
For the convenience of the reader, we briefly explain some basics of gerbes and \(m\)-twisted sheaves mainly after~\cite{MR2388554}.
Let \( p \colon X \to \left( \Sch / \bfk \right) \) be an algebraic stack over \( \bfk \), where we equip \( \left( \Sch / \bfk \right) \) with the fppf topology.
Let \( f \colon \cX \to X \) be a 1-morphism of algebraic stacks over \( \bfk \). The \emph{relative inertia stack} \( I f \) of \( f \) is defined by the following pullback diagram of stacks.
\begin{figure}[H]
    \centering
    \includegraphics[scale=1.2]{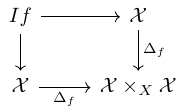}
\end{figure}
More explicitly (see~\cite[\href{https://stacks.math.columbia.edu/tag/034H}{Tag~034H}]{stacks-project} for details) \( If \) is equivalent to the category where
\begin{itemize}
    \item an object is a pair
\( ( x, g ) \) consisting of an object \( x \) of \( \cX \) and an automorphism \( g \) of \( x \) such that \( f ( g ) = 1 _{ f ( x ) } \), and
    \item a morphism from \( ( x, g ) \) to \( ( y, h ) \) is a morphism from \( x \) to \( y \) which is compatible with \( g \) and \( h \) in the obvious manner.
\end{itemize}
Let \( p \colon I f \to \cX \) be one of the projections. In view of the explicit description of  \( I f \) given above, the functor \( p \) simply forgets the automorphism \( g \). In this sense \( I f \) is the universal automorphism group of \( \cX \) relative to \( X \), and actually \( I f \) comes with the canonical structure of a relative group algebraic space over \( \cX \)
(see~\cite[\href{https://stacks.math.columbia.edu/tag/050R}{Tag~050R}]{stacks-project} for more detail).
Recall that the category \( \cX \) admits a standard Grothendieck topology where a collection of morphisms
\(
    ( x _{ i } \to x ) _{ i \in I }
\)
is a covering if and only if the collection of underlying morphisms of schemes is a covering of the site
\(
    \left(
        \Sch / \bfk
    \right)
\).
We will let \( \cC \) denote this site.
Then the functor \( p \colon I f \to \cX \) makes \( I f \) a sheaf over the site \( \cC \). Also we can naturally think of a quasi-coherent sheaf \( F \) on \( \cX \) as a sheaf over \( \cC \), which will be denoted by the same symbols by abuse of notation. Then there is a natural map
\begin{align}
    F \times I f \to F
\end{align}
of sheaves over \( \cC \) defined as follows.
Take an object \( x \in \cC \). Then a section \( g \) of \( I f \) over \( x \) is an automorphism of \( x \), so we obtain the pullback map
\(
    g ^{ \ast } \colon F ( x ) \simto F ( x )
\).
The map
\begin{align}
    F ( x ) \times I f ( x ) \to F ( x )
\end{align}
is simply given by \( ( f, g ) \mapsto g ^{ \ast } f \), and it is a right action by construction. We call it the \emph{inertia action} of the sheaf of groups \( I f \) on the sheaf \( F \).
For a closed subgroup scheme
\(
    i \colon D \hookrightarrow \bGm
\),  a \emph{\(D\)-gerbe} over \( X \), or a gerbe over \( X \) banded by the sheaf of groups \(D _{ \cX }\), is a stack
\(
    f \colon \cX \to X
\)
over
\( X \) which is a gerbe in the usual sense (see, say,~\cite[Definition~2.2.1.2]{MR2388554}) when seen as a stack over the site canonically associated to \( X \) similar to the definition of \( \cC \) above, and an isomorphism \( I f \simeq \cX \times D\) of group stacks over \( \cX \). Through the inertia action, quasi-coherent sheaves on a \( D \)-gerbe comes with a right action of \( D _{ \cX } \).
\begin{definition}[{\(=\)\cite[Definition~3.1.1]{MR2388554}}]\label{definition:twisted sheaf}
    A quasi-coherent sheaf \( F \) on a \( D \)-gerbe \( \cX \) over \( X \) is \( \cX \)-twisted (or \(1\)-twisted) if the right action of \( D _{ \cX } \) on \( F \) coincides with the left action of \( D _{ \cX } \) on \( F \) obtained by pulling back the
    \(
        \mathbb{ G } _{ a, \cX }
        =
        \cO _{ \cX }
    \)-action on \( F \) by the natural map of group schemes
    \(
        D
        _{
            \cX
        }
        \stackrel{ i }{ \hookrightarrow }
        \mathbb{ G }
        _{
            m, \cX
        }
        \to
        \mathbb{ G }
        _{
            a, \cX
        }
    \).
\end{definition}
Let \( \Dhat \) be the character group of \( D \). In general, a quasi-coherent sheaf \( F \) on \( \cX \) admits a decomposition
\begin{align}\label{equation:decomposition}
    F = \bigoplus _{ \chi \in \Dhat } F _{ \chi },
\end{align}
where \( F _{ \chi } \subseteq F \) is the subsheaf of sections on which the inertia action of \( D \) is by the character \( \chi \).
\begin{definition}\label{definition:m-twisted}
    A quasi-coherent sheaf \( F \) on a \( D \)-gerbe \( \cX \) over \( X \) is \( m \)-twisted if \( F = F _{ i ^{ m } } \), where \( i \in \Dhat \) is the standard embedding \( i \colon D \hookrightarrow \bGm \).
    The category of \( m \)-twisted coherent sheaves on \( \cX \) will be denoted by
    \(
        \coh ^{ ( m ) } \cX \subseteq \coh \cX
    \).
\end{definition}
\begin{remark}
    The closed subgroup scheme \(D \hookrightarrow \bGm\) is either
    \(
       \bGm
    \) itself or
    \(
       \mu _{ N }
    \) for some
    \(
       N \ge 1
    \). From now on we assume that
    \(
       D = \mu _{ N }
    \), so that
    \(
       \Dhat
       =
       \bZ / N \bZ
    \).
    For
    \(
       m
       \in
       \bZ
    \)
    we use the standard notation
    \(
       \mbar
       \in
       \bZ / N \bZ
       (
           =
           \Dhat
       )
    \)
    for its residue class, and put
    \(
       \coh ^{ ( \mbar ) } \cX
       =
       \coh ^{ ( m ) } \cX
    \).
\end{remark}
\begin{lemma}\label{lemma:twisted sheaves are acyclic}
    If
    \(
       F
       \in
       \coh ^{ ( \mbar ) } \cX
    \)
    for some
    \(
       \mbar \neq 0
       \in
       \Dhat
    \), then
    \(
       f _{ \ast } F
       =
       0
       \in
       \coh X
    \).
\end{lemma}
\begin{proof}
It is enough to show
\(
   H ^{ 0 } ( \cX \times _{ X } V, F \vert _{ \cX \times _{ X } V } )
   =
   0
\)
for each affine
\(
   V
   \subset
   X
\).
The vanishing directly follows from the definition of
\(
   H ^{ 0 }
\)
for quasi-coherent sheaves on stacks and the definition of twisted sheaves given in~\cref{definition:twisted sheaf}.
\end{proof}
\begin{corollary}\label{corollary:orthogonal decomposition}
    There is an orthogonal decomposition of abelian categories as follows.
    \begin{align}\label{equation:orthogonal decomposition}
        \coh \cX
        =
        \bigoplus _{ \mbar \in \Dhat } \coh ^{ ( \mbar ) } \cX
    \end{align}
\end{corollary}
\begin{proof}
    For    
    \(
       F
       \in
       \coh ^{ ( \mbar ) } \cX       
    \)
    and
    \(
       G
       \in
       \coh ^{ ( \nbar ) } \cX
    \)
    it follows that
    \(
       \cHom _{ \cX } ( F, G )
       \in
       \coh ^{ ( \nbar - \mbar ) } \cX
    \)
    (\cite[Lemma~3.1.1.7~(ii)]{MR2388554}).
    Hence
    \(
        \Hom _{ \cX } ( F, G )       
        =
        0
    \)
    by~\cref{lemma:twisted sheaves are acyclic}.
\end{proof}

%
%
\subsection{\( \cX _{ g } \) and the equivalence of categories}

Let \(g\colon \cG\to\cX_c\) be the gerbe associated to \(\cA_c\) in~\cite[Chapter~V.4.4]{MR0344253}. See also~\cite[Theorem~19]{MR4250478}. In order to make this explicit, let
\(
    p \colon \cX _{ c } \to \left( \Sch / \bfk \right)
\)
be the structure functor of the stack \( \cX _{ c } \).
An object of the category
\(
    \cG
\)
is a triple
\(
    ( x, \cE, \xi )
\), where
\(
    x
\)
is an object of \( \cX _{ c } \), \( \cE \) is a locally free sheaf on the scheme
\(
    p ( x )
\), and
\(
    \xi \colon \cEnd ( \cE ) \simto x ^{ \ast } \cA _{ c }
\)
is an isomorphism of sheaves of algebras on \( p ( x ) \).
A morphism from
\(
    ( x, \cE, \xi )
\)
to
\(
    ( y, \cF, \eta )
\)
is a pair of a morphism
\(
    f \colon x \to y
\)
in \( \cX _{ c } \) and an isomorphism
\(
    \varphi \colon \cE \simto p ( f ) ^{ \ast } \cF
\)
of locally free sheaves on \( p ( x ) \) which is compatible with the isomorphisms
\( \xi, \eta \) in the obvious sense.
The stack \( \cG \) comes with the obvious functor
\(
    g \colon \cG \to \cX _{ c }
\)
which sends the triple \( ( x, \cE, \xi ) \) to \( x \) and the pair \( ( f, \varphi ) \) to \( f \).
The stack \( \cG \) is also equipped with the tautological locally free sheaf \( \cE _{ g } \) together with the universal isomorphism
\(
    \cEnd ( \cE _{ g } ) \simto g ^{ \ast } \cA _{ c }
\).
Moreover \( \cG \) is a \( \bGm \)-gerbe over \( \cX _{ c } \) via \( g \), where the isomorphism
\(
    \bGm \simto \Aut ( x, \cE, \xi )
\)
is given by the following map.
\begin{align}
    \bGm \ni t \mapsto [ \cE \stackrel{ t \cdot }{ \simto } \cE ] \in \Aut ( \cE )
\end{align}
Therefore \( \cE _{ g } \) is a \(1\)-twisted sheaf on the gerbe \( \cG \).
We then obtain functors
\begin{align}
    F \colon \coh^{(1)} \cG \leftrightarrows \coh \cA _{ c } \colon G
\end{align}
given by
\[
    F(\cM)=g_{ \ast }(\cM\otimes_{\cO_{\cG}}\cE^\vee),\quad\quad G(\cN)=g^{ \ast }\cN\otimes_{g^{ \ast }\cA_c}\cE.
\]
Since \(\cN\to FG(\cN)\) is an isomorphism, we see \(F\) is fully faithful.
Next, by~\cite[Lemma~3.1.1.7]{MR2388554}, there is an equivalence of categories as follows.
\begin{align}
    g _{ \ast } \colon \coh ^{ ( 0 ) } \cG \simto \coh \cX _{ c }
\end{align}
See also~\cite[Lemma~12.3.3]{MR3495343} and its proof.
One can easily confirm that the quasi-inverse is given by \( g ^{ \ast } \), so that for each \( \cN \in \coh ^{ ( 0 ) } \cG \) the adjunction counit map
\(
    g ^{ \ast } g _{ \ast } \cN \to \cN
\)
is an isomorphism.
To show \(F\) is essentially surjective, we show \(GF(\cM)\to\cM\) is an isomorphism. We compute \(GF(\cM)=g^{ \ast }g_{ \ast }(\cM\otimes\cE^\vee)\otimes_{g^{ \ast }\cA_c}\cE\). Since \(\cM\) and \( \cE \) are \( 1 \)-twisted, \(\cM\otimes\cE^\vee\) is untwisted by~\cite[Lemma~3.1.1.7]{MR2388554}, so \(g^{ \ast }g_{ \ast }(\cM\otimes\cE^\vee)=\cM\otimes\cE^\vee\). As a result, \(GF(\cM)=\cM\otimes\cE^\vee\otimes_{g^{ \ast }\cA_c}\cE=\cM\).
Finally, suppose that the cohomological Brauer class of \( \cA _{ c } \) comes from a class in \( H ^{ 2 } ( \cX _{ c }, \mu _{ N } ) \). Then there is a \( \mu _{ N } \)-gerbe \( g \colon \cX _{ g } \to \cX _{ c } \), a natural morphism
\(
    \iota \colon \cX _{ g } \to \cG
\),
and the following equivalence of categories (see~\cite[Lemma~3.1.1.12]{MR2388554} and the paragraph before that).
\begin{align}
    \iota ^{ \ast } \colon \coh ^{ ( 1 ) } \cG \simto \coh ^{ ( 1 ) } \cX _{ g }
\end{align}
Thus we have obtained the desired gerbe \( \cX _{ g } \), which is a smooth tame (Artin) stack with finite linearly reductive stabilizers. Notice that if \( \cA \) is a matrix ring over the quotient field of \( X \), then generically, the Brauer class defined by \(\cA_c\) is trivial, and hence \(\cX_g\to\cX_c\) is an isomorphism. In this case \(\coh^{(1)} \cX_g \simeq \coh \cX_c\simeq \coh \cX_g\). This completes the proof of~\cref{theorem:main}.

\section{The Centre}
In this section we prove a special case of the Hochschild-Kostant-Rosenberg isomorphism in degree
\(
   0
\)
for algebraic stacks with finite stabilizers over fields.
This is presumably well known to experts, but for the lack of a suitable reference we include a proof.
Let
\(
   \cX
\)
be an algebraic stack over a field \( \bfk \).
Consider the following Cartesian diagram defining the inertia stack
\(
   I \cX
\), where
\(
    d
    =
    e
    =
    \Delta _{ \cX / \bfk }
\).
\begin{figure}[H]
    \centering
    \includegraphics[scale=1.2]{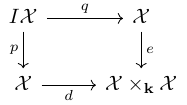}
    \caption{}
    \label{figure:defining square of inertia stack}
\end{figure}
    
\begin{proposition}\label{proposition:HKR in degree 0}
    Let \( \cX \) be a smooth algebraic stack with finite stabilizers over a field \( \bfk \).
    Then there is an isomorphisms of
    \(
       H ^{ 0 }
       \left(
           \cX,
           \cO
           _{
               \cX
           }
       \right)
    \)-modules as follows.
    \begin{align}\label{equation:HKR in degree zero}
        \Hom _{ \cX \times \cX } \left( \Delta _{ \cX \ast } \cO _{ \cX }, \Delta _{ \cX \ast } \cO _{ \cX } \right)
        &
        \simeq
        \bigoplus
        _{
            \substack{
                W \subset I \cX : \text{connected component}\\
                \dim W = \dim \cX
                }
        }
        \Hom _{ \cX }
        \left(
            p _{ \ast } \cO _{ W },
            \cO _{ \cX }
        \right)
    \end{align}
\end{proposition}
\begin{proof}
    Note first that
    \begin{align}
        \Hom _{ \cX \times \cX }
        \left(
            d _{ \ast } \cO _{ \cX }, e _{ \ast } \cO _{ \cX }
        \right)
        \simeq
        \Hom _{ \cX }
        \left(
            e ^{ \ast } d _{ \ast } \cO _{ \cX }, \cO _{ \cX }
        \right).
    \end{align}
    The \emph{derived} functor \( \bL e ^{ \ast } d _{ \ast } \) is represented by the convolution kernel
    \begin{align}
        K \coloneqq \Gamma _{ d \ast } \cO _{ \cX } \ast S _{ \ast } \left( \Gamma _{ e \ast } \cO _{ \cX } \right),
    \end{align}
    where
    \(
        \Gamma _{ f } \colon X \to X \times Y
    \)
    denotes the graph of a morphism \( f \colon X \to Y \),
    \(
        \ast
    \)
    denotes the convolution of kernels, and
    \(
        S \colon \cX \times ( \cX \times \cX )
        \simto
        ( \cX \times \cX ) \times \cX
    \)
    exchanges the factors (see, say,~\cite[Section~5]{MR2244106}. We think of the kernels as right modules).
    In particular,
    \begin{align}\label{equation:FM transform}
        \bL e ^{ \ast } d _{ \ast } \cO _{ \cX }
        \simeq
        \Phi _{ K } ( \cO _{ \cX } )
        \simeq
        \bR \pr _{ 2 \ast } K.
    \end{align}
    Then there is an isomorphism between \( K \) and the structure sheaf of the derived loop space as in~\cref{lemma:derived loop space}
    below. It then follows that
    \begin{align}\label{equation:HKR in degree 0}
        \cH ^{ 0 } ( K ) \simeq  d _{ \ast } \cO _{ \cX } \otimes _{ \cO _{ \cX \times \cX } } e _{ \ast } \cO _{ \cX }
        \simeq
        ( d \circ p ) _{ \ast } \cO _{ I \cX },
    \end{align}
    so that
    \begin{align}\label{equation:key isomorphism}
        e ^{ \ast } d _{ \ast } \cO _{ \cX }
        =
        \cH ^{ 0 } \left( \bL e ^{ \ast } d _{ \ast } \cO _{ \cX } \right)
        \stackrel{\text{\eqref{equation:FM transform}}}{\simeq}
        \cH ^{ 0 } \left( \bR \pr _{ 2 \ast } K \right)
        \simeq
        \pr _{ 2 \ast } \cH ^{ 0 } ( K )
        \stackrel{\text{\eqref{equation:HKR in degree 0}}}{\simeq}
        p _{ \ast } \cO _{ I \cX }.
    \end{align}
    The second-to-the-last isomorphism follows from that \( K \) is supported on the closed substack
    \(
        \Delta _{ \cX } \subset \cX \times \cX
    \), which is finite over \( \cX \).
    Thus we see
    \begin{align}
        \Hom _{ \cX }
        \left(
            e ^{ \ast } d _{ \ast } \cO _{ \cX }, \cO _{ \cX }
        \right)
        &
        \stackrel{\text{\eqref{equation:key isomorphism}}}{\simeq}
        \Hom _{ \cX } ( p _{ \ast } \cO _{ I \cX }, \cO _{ \cX } )
        \simeq
        \bigoplus _{
            \substack{ W \subseteq I \cX\\ \dim W = \dim \cX }
            }
        \Hom _{ \cX } ( p \vert _{ W \ast } \cO _{ W }, \cO _{ \cX } ),
    \end{align}
    where the last isomorphism follows from the smoothness of
    \(
       \cX
    \),
    which implies that
    \(
       \cO
       _{
           \cX
       }
    \)
    is pure.
\end{proof}

\begin{lemma}\label{lemma:derived loop space}
    There is an isomorphism as follows.
    \begin{align}\label{equation:loop space}
        K \simeq d _{ \ast } \cO _{ \cX } \otimes ^{ \bL } _{ \cO _{ \cX \times \cX } } e _{ \ast } \cO _{ \cX }
    \end{align}
\end{lemma}
\begin{proof}
    In this proof, all functors are derived.
    By definition, \( K \) is the following object of \( \derived ^{ \bounded } \coh ( \cX \times \cX ) \); the isomorphism \( S \) disappears due to the obvious symmetry of the object \( \Gamma _{ e \ast } \cO _{ \cX } \).
    \begin{align}
        p _{ 1 4 \ast }
        \left(
            p _{ 1 2 3 } ^{ \ast } \Gamma _{ d \ast } \cO _{ \cX } \otimes p _{ 2 3 4 } ^{ \ast } \Gamma _{ e \ast } \cO _{ \cX }
        \right)
    \end{align}
    We use the following equalities.
    \begin{align}
        \Gamma _{ d } = ( d \times 1 _{ \cX } ) \circ d\\
        \Gamma _{ e } = ( 1 _{ \cX } \times e ) \circ e
    \end{align}
    We also use the following isomorphism of (derived) functors.
    \begin{align}
        p _{ 1 2 3 } ^{ \ast } ( d \times 1 _{ \cX } ) _{ \ast }
        \simeq
        ( d \times 1 _{ \cX } \times 1 _{ \cX } ) _{ \ast } p _{ 1 2 } ^{ \ast }
    \end{align}
    This follows from the following Cartesian square and the flat base change theorem.
    \begin{figure}[H]
        \centering
        \includegraphics[scale=1.2]{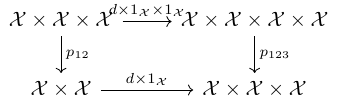}
    \end{figure}
    hence
    \begin{align}
        K \simeq p _{ 1 4 \ast }
        \left(
            ( d \times 1 _{ \cX } \times 1 _{ \cX } ) _{ \ast } p _{ 1 2 } ^{ \ast } d _{ \ast } \cO _{ \cX } \otimes p _{ 2 3 4 } ^{ \ast } \Gamma _{ e \ast } \cO _{ \cX }
        \right)\\
        \simeq
        p _{ 1 4 \ast } ( d \times 1 _{ \cX } \times 1 _{ \cX } ) _{ \ast }
        \left(
            p _{ 1 2 } ^{ \ast } d _{ \ast } \cO _{ \cX } \otimes ( d \times 1 _{ \cX } \times 1 _{ \cX } ) ^{ \ast } p _{ 2 3 4 } ^{ \ast } \Gamma _{ e \ast } \cO _{ \cX }
        \right)\\
        \simeq
        p _{ 1 3 \ast }
        \left(
            p _{ 1 2 } ^{ \ast } d _{ \ast } \cO _{ \cX } \otimes
            { ( 1 \times e ) } _{ \ast } e _{ \ast } \cO _{ \cX }
        \right)\\
        \simeq
        p _{ 1 3 \ast } { ( 1 \times e ) } _{ \ast }
        \left(
            { ( 1 \times e ) } ^{ \ast } p _{ 1 2 } ^{ \ast } d _{ \ast } \cO _{ \cX }
            \otimes
            e _{ \ast } \cO _{ \cX }
        \right)\\
        \simeq
        d _{ \ast } \cO _{ \cX }
        \otimes
        e _{ \ast } \cO _{ \cX }.
    \end{align}
\end{proof}

\begin{remark}\label{remark:HKR}
    We can further compute the summands of the right-hand-side of~\eqref{equation:HKR in degree 0}
    as follows, if the stabilizer groups of
    \(
       \cX
    \)
    are assumed to be \'etale.
    \begin{align}
        \Hom _{ \cX }
        \left(
            p \vert _{ W \ast }
            \cO _{ \cX },
            \cO _{ W }
        \right)
        \simeq
        \Hom _{ W } ( \cO _{ W }, ( p | _{ W } ) ^{ ! } \cO _{ \cX } )
        \simeq
        \End _{ W } ( \cO _{ W } )
        \simeq
        H ^{ 0 } ( W, \cO _{ W } )
    \end{align}
\end{remark}

\begin{example}\label{example:BG}
    Through the isomorphism~\eqref{equation:HKR in degree zero} the algebra structure on the left hand side is transferred to the right hand side. Here we describe it explicitly for the classifying stack
    \(
        \cX
        =
        B G
    \),
    where we let
    \(
       \bfk
    \)
    be a field and
    \(
       G
    \)
    a finite group whose order is coprime to the characteristic of
    \(
       \bfk
    \).
    By abuse of notation, we let
    \(
       G
    \)
    denote the constant group scheme
    \(
       \Spec
       \left(
           \bfk ^{
               G
           }
       \right)
    \).         The isomorphism of Hopf algebras
        \begin{align}\label{equation:OG = kG dual}
        \bfk ^{
            G
        }
        =
        \cO _{
            G
        }
        \simeq
        \left(
            \bfk
            G
        \right) ^{
            \vee
        }
    \end{align}
    implies the following equivalence of symmetric monoidal categories.
    \begin{align}\label{equation:Morita equivalence for BG}
        \coh
        B G
        \simeq
        \coh _{
            G
        }
        \Spec \bfk
        &
        \simeq
        \representations _{
            \bfk
        }
        G
        \simeq
        \module
        \bfk G\\
        \left(
            V,
            V
            \to
            V
            \otimes _{
                \bfk
            }
            \left(
                \bfk ^{
                    G
                }
            \right)
        \right)
        &
        \mapsto
        \left(
            V,
            \bfk G
            \otimes
            V
            \to
            V
        \right)
    \end{align}

    \begin{lemma}\label{lemma:HH0BG vs Z kG}
        There is an isomorphism of
        \(
           \bfk
        \)-algebras
        \(
           \HH ^{ 0 }
           \left(
            B G
           \right)
           \simeq
           \zentrum
           \left(
            \bfk G
           \right)
        \).
    \end{lemma}

    \begin{proof}
        This is a consequence of the Morita invariance of the Hochschild cohomology
        \(
           \bfk
        \)-algebra as follows.
        \begin{align}\label{equation:HH0 BG vs Z kG}
            \HH ^{
                0
            }
            \left(
                B G
            \right)
            =
            \HH ^{
                0
            }
            \left(
                \coh
                B G
            \right)
            \stackrel{\text{\eqref{equation:Morita equivalence for BG}}}{\simeq}
            \HH ^{
                0
            }
            \left(
                \module
                \bfk
                G
            \right)
            =
            \HH ^{ 0 }
            \left(
                \bfk G
            \right)
            \simeq
            \zentrum
            \left(
                \bfk
                G
            \right)
        \end{align}
    \end{proof}
    
    \cref{figure:defining square of inertia stack}
    for
    \(
        \cX
        =
        B G
    \)
    is isomorphic to the following Cartesian square, where the stack
    \(
       \left[
        \frac{
            G
        }{
            G
        }
       \right]
    \)
    is the quotient stack with respect to the conjugacy action. One can directly confirm this by, say, using the explicit description of the inertia stack given in the beginning of~\cref{section:Review of gerbes and twisted sheaves}.
\begin{figure}[H]
    \centering
    \includegraphics[scale=1.2]{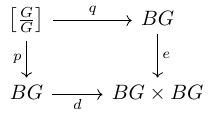}
    \caption{}
    \label{figure:inertia stack of BG}
\end{figure}

As the diagonal map
\begin{align}
    d
    =
    e
    \colon
    B G
    \to
    B G \times B G
    \simeq
    B
    \left(
     G \times G
    \right)
\end{align}
is induced by the diagonal homomorphism
\(
   G
   \to
   G \times G
\),
which is injective, one can easily check that it is representable, finite and flat. This implies the following natural isomorphism of exact functors.
\begin{align}\label{equation:base change formula for BG}
    e ^{
        \ast
    }
    \circ
    d _{
        \ast
    }
    \simeq
    q _{
        \ast
    }
    \circ
    p ^{
        \ast
    }
    \colon
    \coh B G
    \to
    \coh B G
\end{align}
Hence
\begin{align}\label{equation:computation of HH0(BG)}
    \HH ^{
        0
    }
    \left(
        B G
    \right)
    =
    \Hom _{
        B G
        \times
        B G
    }
    \left(
        d _{
            \ast
        }
        \cO _{
            B G
        },
        e _{
            \ast
        }
        \cO _{
            B G
        }
    \right)
    \simeq
    \Hom _{
        B G
    }
    \left(
        e ^{
            \ast
        }
        d _{
            \ast
        }
        \cO _{
            B G
        },
        \cO _{
            B G
        }
    \right)\\
    \stackrel{
        \text{\eqref{equation:base change formula for BG}}
    }{\simto}
    \Hom _{
        B G
    }
    \left(
        q _{
            \ast
        }
        p ^{
            \ast
        }
        \cO _{
            B G
        },
        \cO _{
            B G
        }
    \right)
    \simeq
    \Hom _{
        B G
    }
    \left(
        q _{
            \ast
        }
        \cO _{
            \left[
                \frac{
                    G
                }{
                    G
                }
            \right]
        },
        \cO _{
            B G
        }
    \right).
\end{align}

On the other hand, based on the concrete description of the group structure of
\(
   I \cX
\)
relative to
\(
   \cX
\)
given in~\cite[\href{https://stacks.math.columbia.edu/tag/050R}{Tag~050R}]{stacks-project},
one can directly confirm that the multiplication map
\begin{align}
    I B G
    \times _{
        B G
    }
    I B G
    \to
    I B G
\end{align}
is isomorphic to the following map, where the source of the map is the quotient stack with respect to the diagonal conjugacy action:
\begin{align}\label{equation:group structure of IBG}
    \left[
        \frac{
            G
            \times
            G
        }{
            G
        }
    \right]
    &
    \to
    \left[
        \frac{
            G
        }{
            G
        }
    \right]\\
    (
        g,
        h
    )
    &
    \mapsto
    g h
\end{align}
The group structure on
\(
   I B G
   =
   \left[
       \frac{
           G
       }{
           G
       }
   \right]
\)
given by~\eqref{equation:group structure of IBG}
corresponds to the coalgebra structure
\begin{align}\label{equation:coalgebra structure of q* O IX}
    q _{
        \ast
    }
    \cO _{
     \left[
         \frac{
             G
         }{
             G
         }
     \right]
    }
    \to
    q _{
        \ast
    }
    \cO _{
     \left[
         \frac{
             G
         }{
             G
         }
     \right]
    }
    \otimes
    q _{
        \ast
    }
    \cO _{
     \left[
         \frac{
             G
         }{
             G
         }
     \right]
    }
\end{align}
on
\(
   q _{
       \ast
   }
   \cO _{
    \left[
        \frac{
            G
        }{
            G
        }
    \right]           
   }
\), which in turn yields the
\(
    \bfk
\)-algebra structure on
\(
     \Hom _{
         B G
     }
     \left(
         q _{
             \ast
         }
         \cO _{
             \left[
                 \frac{
                     G
                 }{
                     G
                 }
             \right]
         },
         \cO _{
             B G
         }
     \right)
\).
We then transfer it to
\(
   \HH ^{
       0
   }
   \left(
       B G
   \right)
\)
through the sequence of isomorphisms~\eqref{equation:computation of HH0(BG)}.
    
    \begin{proposition}\label{proposition:BGcentre}
        The
        \(
           \bfk
        \)-algebra structure of
        \(
           \HH ^{
               0
           }
           \left(
            B G
           \right)
        \)
        thus obtained coincides with the standard one.
    \end{proposition}
    
    \begin{proof}
        As
        \(
            \HH ^{
                0
            }
            \left(
             B G
            \right)
        \)
        with the standard
        \(
           \bfk
        \)-algebra structure is isomorphic to
        \(
            \zentrum
            \left(
                \bfk
                G
            \right)
        \)
        under the isomorphism of~\cref{lemma:HH0BG vs Z kG}, we may and will make the comparison on the
        \(
        \bfk
        \)-vector space
        \(
            \zentrum
            \left(
            \bfk
            G
            \right)
        \).

    Under the equivalence~\eqref{equation:Morita equivalence for BG} the object
    \(
        q _{
            \ast
        }
        \cO _{
            \left[
                \frac{
                    G
                }{
                    G
                }
            \right]
        }
        \in
        \coh B G
    \)
    corresponds to
    \(
        \bfk ^{
            G
        }
        \in
       \module
       \bfk
       G
    \)
    on which
    \(
       G
    \)
    is acting by conjugation. Hence
    \begin{align}
        \Hom _{
            B G
        }
        \left(
            q _{
                \ast
            }
            \cO _{
                \left[
                    \frac{
                        G
                    }{
                        G
                    }
                \right]
            },
            \cO _{
                B G
            }
        \right)
        \simeq        
        \Hom _{
            \module
            \bfk
            G
        }
        \left(
            \bfk ^{
                G
            },
            \bfk
        \right)\\
        =
        \Hom _{
            \module \bfk
        }
        \left(
            \bfk ^{
                G
            },
            \bfk
        \right) ^{
            G
        }
        \stackrel{\text{\eqref{equation:OG = kG dual}}}{\simeq}
        \left(
            \bfk
            G
        \right) ^{
            G
        }
        =
        \zentrum
        \left(
            \bfk
            G
        \right).        
    \end{align}
    Moreover we have the following commutative diagram in which all vertical arrows are isomorphisms and the bottom horizontal arrow
    \(
       \mtilde
    \)
    is defined by the commutativity of the diagram.
    \begin{figure}[H]
        \centering
        \includegraphics[scale=1.2]{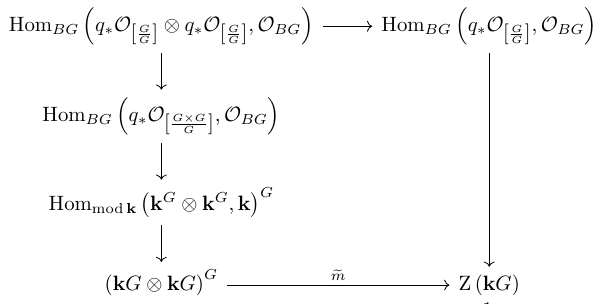}
    \end{figure}    

    One can confirm that the map
    \(
       \mtilde
    \)
    is nothing but the restriction of the multiplication map of the group ring
    \(
       \bfk
       G
    \)
    to the
    \(
       G
    \)-invariant subspaces. This concludes the proof, since what we wanted to show is that the map
    \(
       \mtilde
    \)
    post-composed with the following map coincides with the ring structure of
    \(
       \zentrum
       \left(
           \bfk
           G
       \right)
    \).
    \begin{align}
        \zentrum
        \left(
            \bfk
            G
        \right)
        \otimes _{
            \bfk
        }
        \zentrum
        \left(
            \bfk
            G
        \right)
        =
        \left(
            \bfk
            G
        \right) ^{
            G
        }
        \otimes _{
            \bfk
        }
        \left(
            \bfk
            G
        \right) ^{
            G
        }
        \hookrightarrow
        \left(
            \bfk
            G
            \otimes
            \bfk
            G
        \right) ^{
            G
        }
    \end{align} 
\end{proof}
\end{example}

%
%
\section{Applications and complements}

In this section we restrict to our base field \( \bfk \) having characteristic \(0\). Our main result,~\cref{theorem:main}, allows us to associate a stack to a tame order of global dimension \(2\). In this section, we first discuss the converse construction given in~\cite{MR2018958}. This pair of operations is similar to a Galois connection. The main results of this section are~\cref{corollary:combination1} and~\cref{corollary:combination2}, which describe what happens when these operations are applied in conjunction to create a closure operation.

We first recall how to associate an order to a finite flat scheme cover \(U \to \cX \) of a stack \( \cX \). The cover induces a groupoid of schemes:
    \begin{figure}[H]
        \centering
        \includegraphics[scale=1.2]{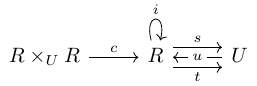}
        \caption{}
        \label{figure:groupoid of schemes}
    \end{figure}
The objects and morphisms in~\cref{figure:groupoid of schemes} are defined as follows.
\begin{align}
    R
    &
    \coloneqq
    U
    \times
    _{
        \cX
    }
    U
    \\
    s
    &
    \coloneqq
    \pr _{ 1 }
    \\
    t
    &
    \coloneqq
    \pr _{ 2 }
    \\
    u
    &
    \coloneqq
    \Delta _{ U / \cX }
    \\
    i
    &
    \coloneqq
    \text{``swap''}
    \\
    c
    &
    \coloneqq
    \left[
        R \times _{ U } R
        =
        U \times _{ \cX } U \times _{ U } U \times _{ \cX } U
        \simto
        U \times _{ \cX } U \times_{ \cX } U
        \xrightarrow[]{\pr _{ 1 3 }}        
        U \times _{ \cX } U
    \right]
\end{align}

Let \(X \) be a scheme.
The category of \(\cO_X \)-bimodules is the category of quasi-coherent sheaves on \(X\times X \) with monoidal structure given by convolution
(see, e.g.,~\cite[p.~6282]{MR2958936} for the definition of convolutions).
An
\emph{\(\cO _{ X }\)-bimodule algebra}
is an algebra in this monoidal category. Note that an \(\cO _{ X }\)-bimodule algebra has possibly distinct structures as a left and right \(\cO _{ X }\)-modules.
An
\emph{\(\cO _{ X }\)-algebra} is an \(\cO _{ X }\)-bimodule algebra in which \(\cO _{ X }\) is central.

We have the diagram of \(\cO _{ U }\)-bimodule algebras as follows, which is dual to the diagram~\cref{figure:groupoid of schemes}.
    \begin{figure}[H]
        \centering
        \includegraphics[scale=1.2]{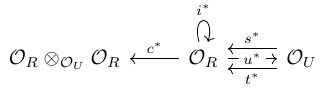}
        \caption{}
        \label{figure:hopfAlgebroid}
    \end{figure}

    If we have a category $\cC$ with pushouts, we define a cogroupoid object in $\cC$ to be a diagram in $\cC$ dual to~\cref{figure:groupoid of schemes}.  For further detail see for example~\cite{MR2553659}.
We define a \emph{Hopf algebroid} over a base scheme \( S \) to be a cogroupoid of quasi-coherent sheaves of
\(
   \cO _{ S }
\)-algebras which will be abbreviated as
\( \cO_R \leftleftarrows \cO_U\). We say this Hopf algebroid is \emph{commutative} if
\( \cO_R \) and \( \cO_U \) are commutative \(\cO _{ S }\)-algebras. The category of
\(\cO_R \leftleftarrows \cO_U\) comodules is symmetric monoidal.
In fact we have the following equivalence
    \begin{figure}[H]
        \centering
        \includegraphics[scale=1.2]{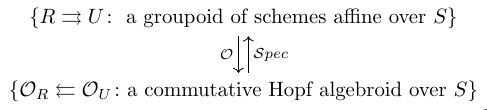}
    \end{figure}
with an equivalence of symmetric monoidal categories as follows.
\begin{align}\label{equation: [ U / R ] and Hopf algebroid}
    \qcoh [ U / R ]
    \simeq
    \Comodule
    \left(
        \cO_R \leftleftarrows \cO_U    
    \right)
\end{align}

Let \(\cA \leftleftarrows \cB\) be a Hopf algebroid over \( S \),
and note that the co-composition morphism defines a co-multiplication
\(
    \cA \otimes_\cB \cA \leftarrow \cA
\).
We say \( \cA \) is cocommutative if this map is symmetric and \( \cB \) is commutative.
Given a Hopf algebroid \(\cA \leftleftarrows \cB\) we define the dual \(\cA^\vee\)
to be the sheaf of left module homomorphisms
\(
   \cHom _{ \cB - \Module }
   \left(
       {} _{ \cB }
       \cA,
       \cB
   \right)
\), where
\( \cA \) is a left \( \cB \)-module by a choice of the two structure morphisms.
In particular, when \(\cA=\cO_R\), we have that \(\cO_R\) is a bimodule and has \(\cO_U\)-module structures induced by \(s^\ast\) and \(t^\ast\).
The construction of
\(
   \cO _{ R } ^{ \vee }
\)
is independent of this choice. We call the  algebra \(\cO_R^\vee\) the convolution algebra of the groupoid
\(R\rightrightarrows U\).
We have the following correspondences
    \begin{figure}[H]
        \centering
        \includegraphics[scale=1.2]{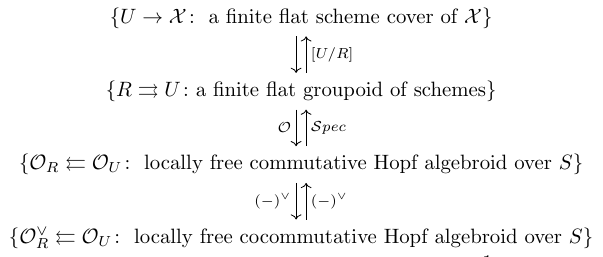}
        \caption{}
        \label{figure:correspondences}
    \end{figure}
We use the notation \(\cA \leftleftarrows \cB\) for   \(\cO_R^\vee  \leftleftarrows \cO_U\).
The dual of a Hopf algebroid is again a Hopf algebroid by the main result of~\cite{MR3606499}, and so there are structure maps \(\cO_R^\vee  \rightrightarrows \cO_U\) dual to \(\cO_R \leftleftarrows \cO_U\),
but our shorthand only focuses on \(\cO_R^\vee  \leftleftarrows \cO_U\).  In our situation, these are the same algebra map, induced by the identity map of the groupoid.  Note that the image of \(\cO_U\) is not in the centre of \(\cO_R^\vee\) in
general.
We have equivalences of monoidal categories as follows.
\begin{align}
    \Comodule \cO_R
    \simeq
    \Module \cO_R^\vee
\end{align}

Note that
\(
    \cO_R
    \simeq
    \cO _{ U \times _{ \cX } U }
\)
in~\cref{figure:hopfAlgebroid} has the structure of a commutative Hopf algebroid.
We now apply the dual functor
\(
    {{ ( - ) }}^\vee
    =
    \cHom_U(-,\cO_U)
\)
to obtain a cocommutative Hopf algebroid structure on
\(
    \cO _R^\vee
\).

The comultiplication gives the category \(\Module \cO_R^\vee\) the structure of a closed symmetric monoidal category.

\begin{theorem}[{\cite{MR2018958}}]\label{theorem:convolutionAlgebra}
    Let \( \cX \) be a stack with a finite flat scheme cover
    \(
       p
       \colon
       U
       \to
       \cX
    \)
    which admits a quasi-projective coarse moduli space \( S \). Then the algebra
    \begin{align}
        \cA
        \coloneqq
        \Hom_U
        \left(
            s_{ \ast } \cO_{U\times_\cX U},
            \cO_U
        \right),
    \end{align}
    which is associated to \(p\) via the correspondence~\cref{figure:correspondences},
    is a cocommutative Hopf algebroid over \(\cO_U\).
    There is an equivalence of symmetric monoidal categories \(\module\cA \simeq \coh \cX \).
    In addition, if \(U\) and \( \cX \) are normal and \( \cX \) is generically tame,
    in the sense that the stabilizer of the generic point is linearly reductive, then the Hopf algebroid \( \cA \) is an order over \( S \) (possibly with centre larger than \(\cO_S\)). 
\end{theorem}

\begin{proof}
    The
    \(
       \cO _{ U }
    \)-bimodule
    \(
        \cO _{ U \times _{ \cX } U }
    \)
    is finite over its centralizer.
    Therefore we can apply~\cite[Proposition~4.4]{MR2018958} to see that there is an equivalence of categories \(\module\cA \simeq \coh \cX \).
    The Hopf algebroid and symmetric monoidal category structures are not treated in~\cite[Theorem~3.3]{MR2018958}, but follow easily from the constructions in that paper.
    For the last statement, \( \cA \) as a sheaf on \( S \) is coherent since both \( S \) and the canonical map \(U \to S\) are finite. It is also torsion free as \(U\) is normal and hence
    \(\cO _{ U }\) is torsion free.
    Lastly, we have that \( \cA \) is a semisimple algebra at the generic point \(\eta\)
    since \(\module \cA_\eta \simeq \coh \cX_\eta \) which is semi-simple since \( \cX \) is generically tame. 
\end{proof}

The following result of Kresch and Vistoli gives conditions that ensure the existence of a finite flat scheme cover.  Quotient stacks are also characterized in
~\cite{MR2108211}.

\begin{corollary}
    Let \( \cX \) be a Deligne--Mumford stack that is separated and of finite type over \( \bfk \) with quasi-projective coarse moduli space
    \(
       S
    \).
    Suppose further that \( \cX \) is a quotient stack (in particular this holds if \( \cX \) is smooth and generically tame).  Then there exists a finite flat scheme cover \(p \colon U \to \cX \).
\end{corollary}

\begin{proof}
The finite flat scheme cover \(p \colon U \to \cX \) exists under the assumptions of the statement (including those in the parentheses) by Kresch and Vistoli~\cite[Theorems~1, 2, and 3]{MR2026412} and Gabber's Theorem~\cite{deJong}.
By~\cite[Theorem~3.3]{MR2018958} the map \(U \to S\) is finite.
\end{proof}

Given the correspondence between stacks \( \cX \) and orders \( \cA \) in~\cref{theorem:main,theorem:convolutionAlgebra}, it is natural to ask how properties of \( \cX \) and \( \cA \) relate. \cref{lemma:structure theorem for DM stacks in dimension 0}--\cref{corollary:structure theorem in dimension 0} describe the correspondence between stacks and algebras in dimension 0, and will be applied to the generic point of our scheme \(X \). In~\cref{corollary:order} we describe how properties of the stack translate to those of the order, the opposite direction of our main theorem.
Lastly, we look at the closure operation given by applying these operations
forward and backward in~\cref{corollary:combination2,corollary:combination2}. 

Given a field \(K\) and a group scheme \( A \) over \(K\), the cocycles in the non-abelian Galois cohomology \(H^2(K,A)\) are described in~\cite{MR0209297} and are given by a Galois extension \(L\) of \(K\) with Galois group \( G \) and maps \(f \colon G\times A \to A\), \(g \colon G\times G \to A\) satisfying the following conditions for all
\(
   a
   \in
   A
\)
and
\(
   r,
   s,
   t
   \in
   G
\).

\begin{itemize}
    \item 
    \(
       a
       \mapsto
       f _{ s } ( a )
    \)
    is an action of \( G \) on \( A \).

    \item
    The following cocycle conditions are satisfied.
    \begin{align}\label{equation:cocycleConditions}
        f_s(f_t(a))
        &
        =
        g_{
            s,
            t
        }
        f_{
            s t
            } ( a )
        g_{
            s,
            t
            }^{-1}
        \\
        f_r(g_{s, t})g_{r,st}
        &
        =
        g_{r,s}g_{rs,t}
    \end{align}
\end{itemize}

\begin{lemma}\label{lemma:structure theorem for DM stacks in dimension 0}
    Let \( \cX \) be a smooth and separated Deligne--Mumford stack over a characteristic \(0\) field \( K \) such that the structure map \(\cX\to\Spec K\) is a coarse moduli space.  Then \( \cX \) is a banded
    \( A \)-gerbe for some finite \'etale group scheme \( A \) over \(K\)
    and is determined by a Galois extension \(L\) of \(K\)
    with Galois group \( G \) and a cocycle \(\sigma=(f,g)\) 
    in the non-abelian Galois cohomology group \(\mathrm{H}^2(K,A)\) as in~\eqref{equation:cocycleConditions}.
    There is a groupoid presenting the stack \( \cX \)
    \begin{figure}[H]
        \centering
        \includegraphics[scale=1.2]{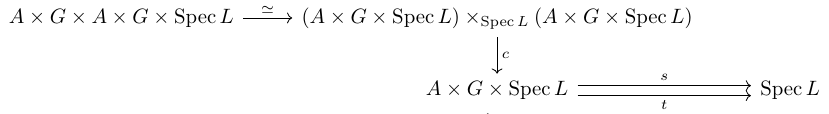}
    \end{figure}
    with convolution, source, and target maps as follows.
    \begin{align}
        (k,s,\ell,t,x)
        &
        \stackrel{c}{\mapsto} (kf_s(\ell) g_{s, t},st,x)
        \label{equation:multiplicationGroupoid}
        \\
        (k,g,x)
        &
        \stackrel{s}{\mapsto}  x
        \\
        (k,g,x)
        &
        \stackrel{t}{\mapsto} gx
    \end{align}
\end{lemma}

\begin{proof} 
Let \(\alpha \colon \cX \to \cXbar\) be the rigidification of \( \cX \). By~\cite[Proposition~2.1]{MR2357471}, the structure map \(\cX \to \Spec K\) factors as \(\cX\xrightarrow{\alpha}\cXbar\xrightarrow{\beta}\Spec K\), where \(\beta\) is a coarse space map which is an isomorphism over a dense open subspace of \(\Spec K\). As a result, \(\beta\) is an isomorphism. By~\cite[Remark A.2 and Example A.3]{MR2427954}, the map \(\alpha\) is an \( A \)-gerbe for some finite \'etale group scheme \( A \) over \( K \).

The classification of banded gerbes by cohomology is in Giraud~\cite{MR0344253} with a concrete realization over fields as nonabelian Galois cohomology in Springer~\cite{MR0209297}.  A groupoid presentation of a gerbe constructed by a cocycle is the same as in the proof of~\cite[Theorem~3.1]{https://doi.org/10.48550/arxiv.math/0212266}, suitably interpreted as Galois cohomology.
\end{proof}

\begin{definition}\label{groupoid-over-field}
  Let \(K\) be a field, let \(L\) be a Galois extension of \(K\) with Galois group \( G \).  Let \( A \) be a finite group scheme over \(K\).  Further assume that \(|A(L)|
  =\deg A\).  Let \(LA\)  be the group algebra
  \(LA(L) = (\cO_{A}\otimes L)^\vee\)
  of the constant group
  \(
     A ( L )
  \)
  with the coefficient field
  \(
     L
  \).
  Let \(\sigma\) be a cocycle in \(H^2(K,A)\) as above determined by \((f,g)\) with
  \(f\) given by the action of \( G \) on \(LA\).
  We define the \emph{crossed product algebra}
  \(LA \rtimes_\sigma G\) as the \(K\)-vector space
  \(
    LA \otimes_K KG
  \)
  with multiplication given by
    \begin{equation}\label{equation:multiplicationAlgebra}
        (a \otimes s) (b \otimes t)
        =
        a f_s(b) g_{s, t} \otimes st
    \end{equation}
 for \(a,b \in LA\) and \(s,t \in G\), as in~\cite[Definition~7.1.1]{Montgomery}.
\end{definition}

\begin{corollary}\label{corollary:structure theorem in dimension 0}
  With hypotheses as in~\cref{lemma:structure theorem for DM stacks in dimension 0}, there is an equivalence of categories
  \[\module (LA \rtimes_\sigma KG) \simeq \coh \cX.\]
\end{corollary}
\begin{proof}
  The description of the multiplication of the groupoid
  in~\eqref{equation:multiplicationGroupoid} induces the multiplication of the convolution
  algebra in~\eqref{equation:multiplicationAlgebra} and so the result follows from
\cref{theorem:convolutionAlgebra}.
\end{proof}

\begin{corollary}\label{corollary:order}
    Let \(R\rightrightarrows U\) be a finite flat groupoid
    of schemes presenting a Deligne--Mumford stack \(\cX = [U/R]\) over a characteristic \(0\) field, with \(U\) normal and coarse moduli space \(X \).
    We further assume that $\cX$ is smooth over the generic point $k(X)$.
   Then \(\cO_R^\vee\) is an order in a semisimple \(\bfk(X)\)-algebra which is Morita equivalent to
    \(\Module (LA \rtimes_\sigma KG)\) as in~\cref{corollary:structure theorem in dimension 0} with \[ \module (LA \rtimes_\sigma KG) \simeq \coh( \cX \times_X \bfk(X)).\]
    Moreover, \(\cO_R^\vee\) is a tame order and \(Z(\cO_R^\vee)\) is the integral closure of \(\cO_X \) in \(\zentrum( LA \rtimes_\sigma KG )\).
    If $\cX$ is smooth, we can also identify the centre in the following way.
 Let \( I \) be the union of the maximal dimensional components of the coarse moduli space of the inertia stack of \( \cX \).
    Then
    \(\cO_I^\vee \coloneqq \zentrum (\cO_R^\vee)\) has \(\cO_I^\vee\simeq \cO_I\) as a coherent sheaf over \(X\).
    
\end{corollary}

\begin{proof}
  The behaviour at the generic point follows from~\cref{corollary:structure theorem in dimension 0}.
  We can localize \(\coh \cX \simeq \module  \cO_R^\vee\) to codimension one
  to show that \(\cO_R^\vee\) is hereditary. Since \(U\) is normal,  \(\cO_R^\vee\) is reflexive.  Since  \(\cO_R^\vee\) is tame, its centre is normal, giving the first description of the centre. The last statement about the centre follows from~\cref{proposition:HKR in degree 0} and~\cref{remark:HKR}.
\end{proof}

If we start with an order, associate a stack using~\cref{theorem:main}, and then take its convolution algebra via~\cref{theorem:convolutionAlgebra}, we obtain an order whose module category is monoidal. The result is~\cref{corollary:combination1} below. Recall from~\cref{definition:m-twisted} that given a \( \mu_N \)-gerbe
\( \cX \) the category of \( i \)-twisted sheaves is \(\qcoh \cX^{(i)}\).
\begin{corollary}\label{corollary:combination1}
    Let \(\cA_1\) be a tame split order of global dimension 2 over a surface \( X \) as in~\cref{definition:nc-surface} with generic point \(\eta\).
    Then there exists a \( \mu_N \)-gerbe
    \(
       \cX
    \)
    over a smooth tame algebraic stack of dimension \( 2 \) whose coarse space is \( X \) for some \(N\) and algebras \(\cA_2,\ldots,\cA_N\) over \( X \) such that
    \begin{enumerate}
    \item \(\module \prod \cA_i \simeq \qcoh \cX \).
    \item \(\module \cA_i \simeq \qcoh \cX^{(i)}\).
    \item \(N\) is the order of \(\cA_{1,\eta}\) in \(\Br \bfk(\eta)\).
    \end{enumerate}
\end{corollary}
As a consequence of this result, given \(\cA_1\) as above, we can associate
\(
   \cAbar
   \coloneqq
   \prod _{ i } \cA _{ i }
\)
which enjoys the geometric properties of being Morita equivalent
to a stack. In particular, we have the following structures for
\(\module \cAbar\):
\begin{itemize}
    \item a symmetric monoidal structure
    \item tangent and cotangent sheaves
    \item Picard group
\end{itemize}
Since  $\cA_{i,\eta}$ and $\cA_{1,\eta}^{\otimes i}$ are Morita equivalent at the generic point, 
it would be interesting to describe a direct relation between them.  
Next, if we start with a stacky surface and finite flat scheme cover, we may form the convolution algebra by~\cref{theorem:convolutionAlgebra}, and then construct the associated stack by~\cref{theorem:main}. The resulting stack $\cXbar$ is described below.  In particular, we note that it would be interesting to determine its dependence on the choice of the finite flat scheme cover.  One
advantage of the stack $\cXbar$ is that it has abelian generic stabilizers.
We define \(\cHH^0(\cX)\) to be the sheaf of algebras over the coarse moduli space given by Hochschild cohomology locally.
\begin{corollary}\label{corollary:combination2}
    Let \( \cX \) be a smooth tame algebraic stack of dimension at most \(2\) which admits a finite flat scheme cover
    \(
        R \rightrightarrows U \to \cX
    \).
    Then the convolution algebra \(\cO_R^\vee\) is a tame order over \(\cHH ^0(\cX)\)
    with global dimension equal to \(\dim\cX \).  Furthermore, 
    there is a stack \(\cXbar\) such that
    \(
        \cXbar = \coprod \cX_i
    \)
    where \(\cX_i\) are
    \(
       \mu _{ n _{ i } }
    \)-gerbes over smooth tame stacks with trivial generic stabilizer
    over components of
    \(\cSpec \cHH ^0(\cX)\).
    
    Furthermore,  \(\qcoh \cX \) is a component of \(\qcoh \cXbar\) as monoidal categories.
    If \( \cX \) has trivial generic stabilizer, then \(\cX \simeq \cXbar\).
\end{corollary}

\begin{proof}
    For the first statement, we have \(\cX = [U/R]\) and we may take \(U \to \cX \) as our finite cover with convolution algebra \(\cO_R^\vee.\)
    This is a tame  order by~\cref{corollary:order}.  We have an equivalence of categories
    \(\qcoh \cX \simeq \module \cO_R^\vee\).  So the global dimension of \(\cO_R^\vee\)
    is the dimension of the coarse moduli space of \( \cX \).
    To build the stack \(\cXbar\) we apply our main construction to each component
    of \(\cO_R^\vee\).
 
    For the last statement, we know that \(\qcoh \cXbar \simeq \qcoh \cX \) as monoidal categories so the reconstruction theorem~\cite[Remark~5.12]{Lurie} for stacks applies to show the result.
\end{proof}

\begin{example}
    Let \(\cX = BG\) for some linearly reductive group scheme \( G \).
    Then the scheme cover
    \(
        \Spec \bfk \to \cX
    \)
    gives us the convolution algebra \(\bfk G\).
    Since \( G \) is linearly reductive we have
    \(
        \bfk G
        \simeq
        \prod _{ i }
        D_i ^{ n_i \times n_i }
    \)
    for some division algebras \(D_i\) with order
    \(
        \ell_i
    \)
    in
    \(
       \Br
       \zentrum
       \left(
        D _{ i }
       \right)
    \).
    On the other hand we have
    \(
        \prod
        _{
            \Irrep
            G
        }
        \bfk
        \simeq
        \zentrum
        \left(
            \prod _{ i }
            D _{ i }
        \right)
        \simeq
        \zentrum
        \left(
            \bfk G
        \right)
        \simeq
        \HH^0(\cX)
    \),
    where, as we saw in~\cref{proposition:BGcentre}, the ring structure of the last term is obtained as the \( \bfk \)-dual of the comultiplication of
    \(
       \cO _{
        I \cX
       }
    \)
    corresponding to the group structure of
    \(
       I \cX
    \)
    over
    \(
       \cX
    \).
    We attach the trivial \(\mu_{\ell_i}\)-gerbe
    \(
       \cX _{ i }
       \coloneqq
       B \mu _{ \ell _{ i } }
    \)
    over the corresponding component of
    \(
        \Spec
        \HH^0(\cX)
    \)
    to each \(D_i\) with
    \(
        \module D_i \simeq \cX_i^{(1)}
    \).
    Then
    \(
        \cXbar = \coprod _{ i } \cX_i
    \).
    Note that \(\cXbar\) has only abelian stabilizer subgroups whereas \( G \) may be non-abelian.
\end{example}

\printbibliography
\end{document}